\documentclass[reqno, 10pt]{amsart}
{
\usepackage{amsmath}
\usepackage[active]{srcltx}
\usepackage{amssymb}
\usepackage[dvips]{graphicx}
\usepackage[margin=1in]{geometry}
\usepackage{xcolor}
\usepackage{hyperref}
\numberwithin{equation}{section}
\newtheorem{Theorem}{Theorem}[section]
\newtheorem{Definition}[Theorem]{Definition}
\newtheorem{Lemma}[Theorem]{Lemma}
\newtheorem{Remark}[Theorem]{Remark}
\newtheorem{ex}[Theorem]{Example}
\newtheorem{prop}[Theorem]{Proposition}
\newtheorem{Corollary}[Theorem]{Corollary}
\newtheorem{Question}[Theorem]{Question}
\setlength{\parskip}{.5em}
  
\usepackage{enumerate}

\begin{document} 
 
\title{Joint ergodicity of piecewise monotone interval maps}
\author[V. Bergelson]{Vitaly Bergelson}

\address[V. Bergelson]{Department of Mathematics\\ Ohio State University \\ Columbus, OH 43210, USA}
\email{vitaly@math.ohio-state.edu}

\author[Y. Son]{Younghwan Son}
\thanks{The second author is supported by the NRF of Korea (NRF-2020R1A2C1A01005446).}
\address[Y. Son]{Department of Mathematics \\  POSTECH \\  Pohang, 37673, Korea}
\email{yhson@postech.ac.kr}
\maketitle

\begin{abstract}
For $i = 0, 1, 2, \dots, k$, let $\mu_i$ be a Borel probability measure on $[0,1]$ which is equivalent to Lebesgue measure $\lambda$ and let $T_i:[0,1] \rightarrow [0,1]$ be $\mu_i$-preserving ergodic transformations.
 We say that transformations $T_0, T_1, \dots, T_k$ are {\em uniformly jointly ergodic} with respect to $(\lambda; \mu_0, \mu_1, \dots, \mu_k)$ if
for any $f_0, f_1, \dots, f_k \in L^{\infty}$, 
\begin{equation*}
    \lim\limits_{N -M \rightarrow \infty} \frac{1}{N-M } \sum\limits_{n=M}^{N-1} f_0 ( T_0^{n} x)  \cdot f_1 (T_1^n x) \cdots f_k (T_k^n x)  =
    \prod_{i=0}^k \int f_i \, d \mu_i   \quad \text{ in } L^2(\lambda).
\end{equation*}

We establish convenient criteria for uniform joint ergodicity and obtain numerous applications, most of which deal with interval maps. 
Here is a description of one such application. Let $T_G$ denote the Gauss map, $T_G(x) = \frac{1}{x} \, (\bmod  \, 1)$,  and, for $\beta >1$, let $T_{\beta}$ denote the $\beta$-transformation defined by $T_{\beta} x = \beta x \, (\bmod  \,1)$.
Let $T_0$ be an ergodic interval exchange transformation. Let $\beta_1 , \cdots , \beta_k$ be distinct real numbers with $\beta_i >1$ and assume that $\log \beta_i \ne \frac{\pi^2}{6 \log 2}$ for all $i = 1, 2, \dots, k$. Then for any $f_{0}, f_1, f_{2}, \dots, f_{k+1} \in L^{\infty} (\lambda)$, 
\begin{equation*}
\begin{split} 
\lim\limits_{N -M \rightarrow \infty} \frac{1}{N -M } \sum\limits_{n=M}^{N-1} 
&   f_{0} (T_0^n x) \cdot  f_{1} (T_{\beta_1}^n x) \cdots f_{k} (T_{\beta_k}^n x) \cdot  f_{k+1} (T_G^n x)  \\
 &= \int f_{0} \, d \lambda \cdot \prod_{i=1}^k \int f_{i} \, d \mu_{\beta_i} \cdot  \int f_{k+1} \, d \mu_G  \quad \text{in } L^{2}(\lambda).
\end{split}
\end{equation*}

We also study the phenomenon of {\em joint mixing}. 
Among other things we establish joint mixing for skew tent maps and for restrictions of finite Blaschke products to the unit circle.
\end{abstract}

\section{Introduction}
\label{sec:Intro}

A measure preserving transformation $T$ on a probability space $(X, \mathcal{B}, \mu)$ is called {\em ergodic} if the only measurable sets $B$ with $T^{-1} B = B$ satisfy $\mu(B) = 0$ or $\mu(B) = 1$. Transformation $T$ naturally induces an isometric operator on $L^p$ spaces, $1 \leq p \leq \infty$, by the formula $Tf := f \circ T$.
The classical von Neumann's mean ergodic theorem asserts that $T$ is ergodic if and only if for any $f \in L^{\infty}(X)$, 
\[ \lim_{N-M \rightarrow \infty} \frac{1}{N-M} \sum_{n=M}^{N-1} T^n f = \int f \, d \mu \, \text{ in } L^2(\mu).\]

The following definition of {\em joint ergodicity}, which was introduced in \cite{BeBe1}, is an extension of notion of ergodicity to a finite family of measure preserving transformations on $(X, \mathcal{B}, \mu)$.

\begin{Definition}
Measure preserving transformations  $T_1, \dots, T_k$  on a probability space $(X, \mathcal{B}, \mu)$  are called {\em uniformly jointly ergodic} if  for any $f_1, \dots f_k \in L^{\infty} (\mu)$, 
\begin{equation*}
\lim_{N -M \rightarrow \infty} \frac{1}{N -M} \sum_{n=M}^{N-1} T_1^{n} f_1 \cdot T_2^n f_2 \cdots T_k^n f_k = \prod_{i=1}^k \int f_i \, d \mu \quad \text{ in } L^2(\mu). 
\end{equation*}
\end{Definition}

The following theorem, which was proved in \cite{BeBe1}, provides a useful criterion for joint ergodicity under the assumption that the transformations $T_i$ are invertible and commuting. 
\begin{Theorem}
\label{Intro:ThmJENS1}
Let $T_1, \dots, T_k$ be commuting, invertible measure preserving transformations on $(X, \mathcal{B}, \mu)$.  
The following conditions are equivalent:
\begin{enumerate}[(i)]
\item $T_1, \dots, T_k$ are uniformly jointly ergodic. 
\item \begin{enumerate}
\item $T_1 \times T_2 \times \cdots \times T_k$ is ergodic with respect to $\mu \times \mu \times \cdots \times \mu$
\item $T_i^{-1} T_j$ is ergodic for any $i \ne j$.
\end{enumerate}
\end{enumerate} 
\end{Theorem}

We list now some examples illustrating the phenomenon of joint ergodicity, which can be verified with the help of Theorem \ref{Intro:ThmJENS1}.

\begin{ex}
\label{ex:intro}
\mbox{}

\begin{enumerate}
\item  Any finite family of independent translation on $d$-dimensional torus $\mathbb{T}^d$ is uniformly jointly ergodic. More formally,
let $\alpha_1, \dots, \alpha_k \in \mathbb{T}^d$ be such that for any non-zero $k$-tuple $(m_1, \dots, m_k) \in \mathbb{Z}^k$, $\sum_{i=1}^k m_i \alpha_i \notin \mathbb{Q}^k$.  
Let $T_i: \mathbb{T}^d \rightarrow \mathbb{T}^d$ be defined by $T_i x = x+ \alpha_i$ for $1 \leq i \leq k$. Then $T_1, \dots, T_k$ are uniformly jointly ergodic on $(\mathbb{T}^d, \mathcal{B}, \lambda)$, where $\mathcal{B}$ is the Borel $\sigma$-algebra and $\lambda$ is the Lebesgue measure.
\item (cf.  \cite[Theorem 4.11]{Furbook} and \cite[Theorem 3.1]{FKO}) Let $T$ be a weakly mixing transformation\footnote{A measure preserving transformation $T$ on $(X, \mathcal{B}, \mu)$ is called
{\em weakly mixing} if $T \times T$ is ergodic on the product space $X \times X$. 
It is well-known and not hard to verify that this condition is equivalent to  
        \[ \lim_{N-M \rightarrow \infty}  \frac{1}{N-M} \sum_{n=M}^{N-1} | \mu(A \cap T^{-n} B) - \mu(A) \mu(B) | =0 \quad \text{for any } A, B \in \mathcal{B}.\]} on a probability space $(X, \mathcal{B}, \mu)$. Then $T^{a_1}, \dots, T^{a_k}$ are uniformly jointly ergodic on  $(X, \mathcal{B}, \mu)$ for any non-zero distinct integers $a_1, \dots, a_k$. 
\item (\cite[Theorem 1]{Berg}) Let $(T_t)_{t \in \mathbb{R}}$ be a continuous measure preserving ergodic flow on a probability space $(X, \mathcal{B}, \mu)$. Fix $k \in \mathbb{N} \, (:= \{1 ,2 , 3, \dots \})$. Then for almost every $k$-tuple $(a_1, \dots, a_k) \in \mathbb{R}^k$, $T_{a_1}, \cdots, T_{a_k}$ are uniformly jointly ergodic on $(X, \mathcal{B}, \mu)$.
\end{enumerate}
\end{ex}

The goal of this paper is to demonstrate that the phenomenon of joint ergodicity has a quite wide range and, in particular, takes place in certain situations when the involved measure preserving transformations are not necessarily commuting or invertible and have different invariant measures.

\begin{Definition}
\label{Int:Def1.1}
Let $(X, \mathcal{B}, \mu)$ be a probability space. Let $T_1, \dots, T_k$ be measurable transformations on $(X, \mathcal{B})$ such that for each $i$, there is a $T_i$-invariant probability measure $\mu_i$, which is equivalent to $\mu$.
Then transformations $T_1, \dots, T_k$ are said to be {\em uniformly jointly ergodic with respect to $(\mu; \mu_1, \dots, \mu_k)$} if for any $f_1, \dots, f_k \in L^{\infty}$,
\begin{equation}
\label{Int:eq1.2}
    \lim\limits_{N -M \rightarrow \infty} \frac{1}{N -M } \sum\limits_{n=M}^{N-1}  T_1^{n} f_1 \cdot T_2^n f_2 \cdots T_k^n f_k = \prod_{i=1}^k \int f_i \, d \mu_i \quad \text{ in } L^2(\mu).
\end{equation}
\end{Definition}

\begin{Remark}
\label{equivalentmeasures}
Let $(X, \mathcal{B}, \mu)$ be a probability space.
Let $\nu$ be a probability measure on $(X, \mathcal{B})$ such that $\nu$ is equivalent to $\mu$. 
Then $L^{\infty} (\mu) = L^{\infty} (\nu)$. 
Moreover, it is not hard to check that $T_1, \dots, T_k$ are uniformly jointly ergodic with respect to  $(\mu; \mu_1, \dots, \mu_k)$ if and only if $T_1, \dots, T_k$ are uniformly jointly ergodic with respect to  $(\nu; \mu_1, \dots, \mu_k)$ (see Lemma \ref{Lem:Conv:Mean} below).
\end{Remark}

We list now some classical examples of measure preserving transformations which will be used below in examples illustrating the phenomenon of joint ergodicity as manifested by formula \eqref{Int:eq1.2}. 

\begin{enumerate}
\item  For $\beta >1$, define $T_{\beta}: [0,1] \rightarrow [0,1]$ by $T_{\beta} x = \beta x \bmod \, 1$. It is a classical fact that there exists a unique $T_{\beta}$-invariant ergodic probability measure $\mu_{\beta}$ which is equivalent to the Lebesgue measure $\lambda$.  
Indeed, if $\beta$ is an integer $b \geq 2$, then $T_b x = bx \, \bmod \, 1$ preserves Lebesgue measure $\lambda$ and is ergodic. 
If $\beta \notin \mathbb{N}$, R\'{e}nyi (\cite{Re}) showed that  there exists a unique ergodic $T_{\beta}$-invariant probability measure $\mu_{\beta}$ which satisfies $1 - \frac{1}{\beta} \leq \frac{d \mu_{\beta}}{d \lambda} \leq \frac{1}{1 -1 / \beta}$. (See also \cite{Ge} and \cite{Pa0}.) For an integer $b \geq 2$, $T_b$ is called times $b$ map and for a non-integer $\beta > 1$, $T_{\beta}$ is called a beta transformation. (Sometimes $T_{\beta}$ is also called a R\'{e}nyi transformation.) 
\item  Gauss map $T_G: [0,1] \rightarrow [0,1]$ is defined by 
$$T_G(x) = \begin{cases} 
      \frac{1}{x} \, \bmod 1, & x\ne 0, \\
      0, & x=0.
   \end{cases}
$$  
Let $\mu_G$ denote the Gauss measure defined by $\mu_G (A) = \frac{1}{ \log 2} \int_A \frac{1}{1+x} dx$ for any measurable set $A \subset [0,1]$.
It is well-known that  $T_G$ is $\mu_G$-preserving and ergodic. (See \cite{Kno} and \cite{Ryll}.)
\item Interval exchange transformations. 
For a probability vector $\mathbf{a} = (a_1, \dots, a_n)$ ($a_i > 0$ and $\sum_{i=1}^n a_i = 1$), let 
$\beta_0(\mathbf{a}) = 0$ and $ \beta_i (\mathbf{a}) =  a_1 + \cdots + a_i$ for $1 \leq i \leq n.$  
Write $I_i =  [\beta_{i-1}(\mathbf{a}), \beta_{i} (\mathbf{a}) ) \, \text{ for } i = 1, 2, \dots, n.$
Let $\pi$ be a permutation of the symbols $\{1, 2, \dots, n\}$ and let
 $\mathbf{a}^{\pi} = (a_{\pi^{-1}(1)}, \dots, a_{\pi^{-1}(n) } )$. 
An $(\mathbf{a}, \pi)$-interval exchange transformation is defined by 
\[ T_{\mathbf{a}, \pi} (x) = x - \beta_{i-1} (\mathbf{a}) + \beta_{\pi(i-1)} (\mathbf{a}^{\pi}), \quad x \in I_i.\]
Thus, $T_{\mathbf{a}, \pi}$ rearranges intervals $I_1, \dots, I_n$ so that the subinterval at position $i$ is moved to position $\pi(i)$. Interval exchange transformations preserve Lebesgue measure since they are piecewise isometries.
Also, it is well known that interval exchange transformations have zero entropy. For example, it follows from Rokhlin's formula. (See Theorem \ref{Rokhlin} below).

\end{enumerate}

The following theorem, which is a rather special case of a more general result that we will prove in Section \ref{sec:jointergodicityPM} below, provides an example of uniformly jointly ergodic transformations in the sense of Definition \ref{Int:Def1.1}.

\begin{Theorem}
\label{Int:Thm:1.2}
Let $T_0$ be  an ergodic interval exchange transformation.  For any $\beta > 1$ with $\log \beta \ne \frac{\pi^2}{6 \log 2} $, the transformations $T_0, T_{\beta}, T_G$ are uniformly jointly ergodic: for any bounded measurable functions $ f_1, f_2, f_3 \in L^{\infty} (\lambda)$,
\begin{equation*}
\lim_{N - M \rightarrow \infty} \frac{1}{N-M} \sum_{n=M}^{N-1} T_0^n f_1 \cdot T_{\beta}^n f_2 \cdot T_G^n f_3 = \int f_1 \, d \lambda \, \int f_2 \,  d \mu_{\beta} \int f_3 \, d \mu_G \quad \text{ in } L^2 (\lambda)
\end{equation*}
\end{Theorem}

\begin{Remark}
Theorem \ref{Intro:thm:L^2_joint} (see below), of which Theorem \ref{Int:Thm:1.2} is a special case, guarantees joint ergodicity of transformations which have different values of entropy. 
This explains the constraint $h(T_{\beta}) =\log \beta \ne h(T_G) = \frac{\pi^2}{6 \log 2}$. (As we have already mentioned above, interval exchange transformations have zero entropy.) We do not know whether Theorem \ref{Int:Thm:1.2} holds when $\log \beta  = \frac{\pi^2}{6 \log 2}$. (See Question \ref{Que:sameent} below.)
\end{Remark}

The following theorem, which is a common generalization of Theorem 3.1 in \cite{BeBe1} and Theorem 2.1 in \cite{BeBe2}, provides convenient criteria for uniform joint ergodicity. 

\begin{Theorem}[Theorem \ref{criterion:JE}]
\label{Intro:ThmJENS}
Let $(X, \mathcal{B}, \mu)$ be a probability space. Let $T_1, \dots, T_k$ be measurable transformations on $(X, \mathcal{B})$ such that for each $i$, there is a $T_i$-invariant probability measure $\mu_i$, which is equivalent to $\mu$.
The following conditions are equivalent:
\begin{enumerate}[(i)]
\item $T_1, \dots, T_k$ are uniformly jointly ergodic with respect to $(\mu; \mu_1, \dots, \mu_k)$: \\
for any bounded measurable functions $f_1, \dots f_k$, 
\begin{equation*}
\lim_{N-M \rightarrow \infty} \frac{1}{N - M} \sum_{n=M}^{N-1} T_1^{n} f_1 \cdot T_2^n f_2 \cdots T_k^n f_k = \prod_{i=1}^k \int f_i \, d \mu_i \quad \text{ in } L^2(\mu). 
\end{equation*} 
\item for any bounded measurable functions $f_0, f_1, \dots f_k$, 
\begin{equation*}
\lim_{N-M \rightarrow \infty} \frac{1}{N - M} \sum_{n=M}^{N-1} \int f_0 \cdot T_1^{n} f_1 \cdot T_2^n f_2 \cdots T_k^n f_k \, d \mu = \int f_0 \, d \mu \cdot \prod_{i=1}^k \int f_i \, d \mu_i. 
\end{equation*} 
\item \begin{enumerate}
\item $T_1 \times T_2 \times \cdots \times T_k$ is ergodic with respect to $\mu_1 \times \mu_2 \times \cdots \times \mu_k$
\item for  any bounded measurable functions $f_1, \dots f_k$,
\begin{equation*}
\lim_{N-M \rightarrow \infty} \frac{1}{N-M} \sum_{n=M}^{N-1}  \int T_1^{n} f_1 \cdot T_2^n f_2 \cdots T_k^n f_k \, d \mu = \prod_{i=1}^k \int f_i \, d \mu_i .
\end{equation*}
\end{enumerate}
\item \begin{enumerate}
\item $T_1 \times T_2 \times \cdots \times T_k$ is ergodic with respect to $\mu_1 \times \mu_2 \times \cdots \times \mu_k$
\item there is $C>0$ such that for any $A_1, A_2, \dots, A_k \in \mathcal{B}$
\begin{equation*}
\limsup_{N-M \rightarrow \infty} \frac{1}{N-M} \sum_{n=M}^{N-1} \mu (T_1^{-n} A_1 \cap T_2^{-n} A_2 \cap \cdots \cap T_k^{-n} A_k) \leq C \prod_{i=1}^k \mu_i(A_i) .
\end{equation*}
\end{enumerate}
\end{enumerate} 
\end{Theorem}

Theorem \ref{Intro:ThmJENS} will allow us to prove joint ergodicity for a rather wide class of interval maps which we will presently introduce. This class, which we will denote by $\mathcal{T}$, contains, in particular, ergodic interval exchange transformations as well as $T_G$ and $T_{\beta}$ for $\beta > 1$. 

Let $X=[0,1]$ and let $\lambda$ be the Lebesgue measure on $(X, \mathcal{B})$, where $\mathcal{B}$ is the Borel $\sigma$-algebra.
The class $\mathcal{T}$ is comprised of measure preserving systems $(X, \mathcal{B}, \mu, T)$, where $\mu$ is a probability measure on $X = [0, 1]$ such that for some $c \geq 1$, $\frac{1}{c} \mu (A) \leq \lambda (A) \leq c \mu (A)\,  \text{for all } A \in \mathcal{B}$, and $T: [0,1] \rightarrow [0,1]$ is a $\mu$-ergodic\footnote{Clearly, in light of our assumptions, such $T$-invariant measure $\mu$ is unique.}, piecewise monotone map, which is naturally equipped with a (finite or countable) partition  $\mathcal{A}$ of $[0,1]$ into the intervals with disjoint interiors such that for any interval $I \in \mathcal{A}$, the map $T|_{\text{int}(I)}$ is continuous and strictly monotone.
In addition, we will assume that 
\begin{enumerate}
\item $\mathcal{B} = \bigvee_{j=0}^{\infty} T^{-j} \sigma( \mathcal{A}) \mod \lambda$, 
where $\sigma (\mathcal{A})$ denotes the sub-$\sigma$-algebra generated by $\mathcal{A}$. 
( i.e., partition $\mathcal{A}$ generates the $\sigma$-algebra $\mathcal{B}$.) 

\item Entropy of the partition $\mathcal{A}$ is finite: $H(\mathcal{A}) = - \sum_{I \in \mathcal{A}} \mu(I) \log \mu (I) < \infty$. 
\end{enumerate}
For convenience, we will be denoting members of $\mathcal{T}$ by $(T, \mu, \mathcal{A})$.


\begin{Remark}
If a measure preserving transformation $T$ on a Lebesgue space has zero entropy, then it is invertible (see Corollary 4.14.3 in \cite{Walt}.) On the other hand, an invertible transformation with $\mathcal{B} = \bigvee_{j=0}^{\infty} T^{-j} \sigma( \mathcal{A})$ and $H(\mathcal{A}) < \infty$ has zero entropy (see Corollary 4.18.1 and Remarks on p.96 in \cite{Walt}.) Thus $(T, \mu, \mathcal{A}) \in \mathcal{T}$ is invertible if and only if it has zero entropy.  
\end{Remark}



We provide now examples of transformations belonging (along with appropriate partitions) to the class $\mathcal{T}$. (Note that these examples include the transformations appearing in Theorem \ref{Int:Thm:1.2} above.)

\begin{enumerate}
\item  Ergodic interval exchange transformations. 
Recall that for a fixed permutation $\pi$ of $1, 2, \dots, n$ and a probability vector $\mathbf{a} = (a_1, \dots, a_n)$,
an $(\mathbf{a}, \pi)$-interval exchange transformation is given by 
\[ T_{\mathbf{a}, \pi} (x) = x - \beta_{i-1} (\mathbf{a}) + \beta_{\pi(i-1)} (\mathbf{a}^{\pi}), \quad x \in I_i,\]
where $I_i, (1 \leq i \leq n)$ forms a partition of the interval $[0,1]$.
Thus, any ergodic interval exchange transformation $T_{\mathbf{a}, \pi}$ has a natural partition $\mathcal{A}_{T_{\mathbf{a}, \pi}}$ such that $(T_{\mathbf{a}, \pi}, \lambda, \mathcal{A}_{T_{\mathbf{a}, \pi}}) \in \mathcal{T}$. 

\item Times $b$ maps $T_b x = bx \, \bmod 1$ ($b \in \mathbb{N}$, $b \geq 2$):
\begin{equation*}
\mathcal{A}_b = \left\{ \left[0, \frac{1}{b} \right), \cdots, \left[ \frac{b-1}{b}, 1 \right)  \right\}. 
\end{equation*}
\item $\beta$-transformations $T_{\beta} x = \beta x \, \bmod  \, 1$ ($\beta >1$, $\beta \notin \mathbb{N}$):
\begin{equation*} 
\mathcal{A}_{\beta} = \left\{ \left[ 0, \frac{1}{\beta} \right), \cdots, \left[ \frac{[\beta]-1}{\beta}, \frac{[\beta]}{\beta} \right),  \left[ \frac{[\beta]}{\beta}, 1 \right)  \right\}.
\end{equation*}

\item Gauss map $T_G x = \frac{1}{x} \, \bmod \, 1$:
\begin{equation*}
\mathcal{A}_G = \left\{  \left( \frac{1}{n+1}, \frac{1}{n} \right]: n \in \mathbb{N} \right\}.
\end{equation*}
\end{enumerate}


We are going now to formulate a general result, which provides convenient sufficient conditions for joint ergodicity of transformations belonging to class $\mathcal{T}$. But first we need to introduce one more definition.
\begin{Definition}
We will say that a system $(T, \mu, \mathcal{A})$ belonging to $\mathcal{T}$ has {\em property B} if there exist a constant $c > 0$ and a sequence $(b_n)$ of positive real numbers with $b_n \rightarrow 0$ such that  for any natural number $l$ and for any non-negative integer $n$,  
if $A \in \sigma\left(\bigvee_{i=0}^{l-1} T^{-i} \mathcal{A} \right), B \in \mathcal{B}$, then
\[ \mu (A \cap T^{-(n+l)} B) \leq c \mu(A) \mu(B) + b_n.\]
\end{Definition}


We are now in position to formulate one of the main results of this paper.

\begin{Theorem}[Theorem \ref{thm:L^2_joint}]
\label{Intro:thm:L^2_joint}
For $i = 0, 1, 2, \dots, k$, let $\mu_i$ be a probability measure on the measurable space $([0,1], \mathcal{B})$ and let  $T_i:[0,1] \rightarrow [0,1]$ be $\mu_i$-preserving ergodic transformations such that
\begin{enumerate}
\item for $i =0, 1, \dots, k$, there is a partition $\mathcal{A}_i$ with  $(T_i, \mu_i, \mathcal{A}_i) \in \mathcal{T}$,
\item for $i = 1, 2, \dots, k$, $(T_i, \mu_i, \mathcal{A}_i)$ satisfies property B,
\item $h(T_0) < h(T_1) < \cdots < h(T_k)$, where $h(T_i)$ denotes the entropy of $T_i$.
\end{enumerate} 
Suppose that $T_0 \times T_1 \times \cdots \times T_k$ is ergodic on $(X^{k+1}, \otimes_{i=0}^k \mu_i)$. 
Then for any probability measure $\nu$, which is equivalent to the Lebesgue measure $\lambda$,
transformations $T_0, T_1, \dots, T_k$ are uniformly jointly ergodic with respect to $(\nu; \mu_0, \mu_1, \dots, \mu_k)$: 
for any $f_0, f_1, \dots, f_k \in L^{\infty}$, 
\begin{equation*}
    \lim\limits_{N -M \rightarrow \infty} \frac{1}{N-M } \sum\limits_{n=M}^{N-1}  T_0^{n} f_0 \cdot T_1^n f_1 \cdots T_k^n f_k = \prod_{i=0}^k \int f_i \, d \mu_i \quad \text{ in } L^2(\nu).
\end{equation*}
\end{Theorem}

We illustrate Theorem \ref{Intro:thm:L^2_joint} by the following corollary, which includes Theorem \ref{Int:Thm:1.2} as a special case.
\begin{Theorem}
\label{Intro-Thm-JE-Entropy2}
Let $T_0$ be an ergodic interval exchange transformation. Let $\beta_1 , \cdots , \beta_k$ be distinct real numbers with $\beta_i >1$. We assume that $\log \beta_i \ne \frac{\pi^2}{6 \log 2}$ for all $i$. Then for any $f_{0}, f_1, f_2, \dots, f_{k+1} \in L^{\infty} (\lambda)$, 
\begin{equation*}
\begin{split} 
\lim\limits_{N -M \rightarrow \infty} \frac{1}{N -M } \sum\limits_{n=M}^{N-1} 
&  T_0^n f_{0} \cdot T_{\beta_1}^n f_{1} \cdots T_{\beta_k}^n f_{k}  \cdot T_G^n f_{k+1} \\
 &= \int f_{0} \, d \lambda  \cdot \prod_{i=1}^k \int f_{i} \, d \mu_{\beta_i} \cdot  \int f_{k+1} \, d \mu_G \quad \text{in } L^{2}(\nu),
\end{split}
\end{equation*}
where $\nu$ is any probability measure equivalent to the Lebesgue measure $\lambda$.
\end{Theorem}

Since many transformations belonging to $\mathcal{T}$ are strongly mixing\footnote{A measure preserving transformation $T$ is called
 {\em strongly mixing} if for any $A, B \in \mathcal{B}$, 
        \[ \lim_{n \rightarrow \infty}  \mu(A \cap T^{-n} B) = \mu(A) \mu(B).\]}, it makes sense  also to introduce the notion of joint mixing. 
\begin{Definition}
\label{Int:Def:JM}
Let $(X, \mathcal{B}, \mu)$ be a probability space. Let $T_1, \dots, T_k$ be measurable transformations on $(X, \mathcal{B})$ such that for each $i$, there is a $T_i$-invariant probability measure $\mu_i$, which is equivalent to $\mu$.
Transformations $T_1, \dots, T_k$ are said to be {\em jointly mixing with respect to $(\mu; \mu_1, \dots, \mu_k)$} if for any $A_0, \dots, A_k \in \mathcal{B}$,
\begin{equation*}
 \lim_{n \rightarrow \infty} \mu (A_0 \cap T_1^{-n}A_1 \cap \cdots \cap T_k^{-n}A_k) =  \mu (A_0) \prod_{i=1}^k \mu_i(A_i).
 \end{equation*}
\end{Definition}

\begin{Remark}
\label{rem:mixje}
Joint mixing implies uniform joint ergodicity.
Indeed it is not hard to see that if transformations $T_1, \dots, T_k$ are jointly mixing, then they satisfy the condition $(iv)$ of Theorem \ref{Intro:ThmJENS}. 
\end{Remark}

\begin{Remark}
It is not hard to see that if $\mu$ and $\nu$ are equivalent, then 
$T_1, \dots, T_k$ are jointly mixing with respect to  $(\mu; \mu_1, \dots, \mu_k)$ if and only if $T_1, \dots, T_k$ are jointly mixing with respect to  $(\nu; \mu_1, \dots, \mu_k)$ (see Lemma \ref{Lem:Conv:Weak}). 
\end{Remark}

In addition, we also obtain a sufficient condition for joint mixing of transformations in $\mathcal{T}$. To formulate this result, we need M. Rosenblatt's notion of $\alpha$-mixing {\cite{Ro}}, which is somewhat stronger than property B.

\begin{Definition}
$(T, \mu, \mathcal{A}) \in \mathcal{T}$ is {\em $\alpha$-mixing}
if there exists a sequence $\alpha(n), n=1, 2, \dots,$ of non-negative real numbers with $\lim\limits_{n \rightarrow \infty} \alpha(n) = 0$  such that for any natural number $l$, for any non-negative integer $n$, if $A \in \sigma \left( \bigvee_{j=0}^{l-1} T^{-j} \mathcal{A} \right)$ and  $B \in \mathcal{B}$,
\begin{equation*} 
| \mu (A \cap T^{-(n+l)} B) - \mu(A) \mu(B)| \leq \alpha(n).
\end{equation*}
\end{Definition}

Examples of transformations with $\alpha$-mixing include $T_G$ and $T_{\beta}$ $(\beta >1)$ (see Section  \ref{sec:examples}.)

\begin{Remark}
If $(T,\mu, \mathcal{A}) \in \mathcal{T}$ is $\alpha$-mixing, then it is {\em exact} (see Theorem \ref{exactness:new}). Thus $h(T_i) > 0$ for all $i = 1, 2, \dots, k$ in the following theorem.
\end{Remark}

\begin{Theorem}[Theorem \ref{Thm:JM2:new}]
\label{Int:Thm:JM2:new}
For $i = 0, 1, 2, \dots, k$, let $\mu_i$ be a probability measure on the measurable space $([0,1], \mathcal{B})$ and let  $T_i:[0,1] \rightarrow [0,1]$ be $\mu_i$-preserving ergodic transformations such that
\begin{enumerate}
\item for $i = 0, 1, 2, \dots, k$, there is a partition $\mathcal{A}_i$ with  $(T_i, \mathcal{A}_i, \mu_i) \in \mathcal{T}$,
\item for $i = 1, 2, \dots, k$, $(T_i, \mathcal{A}_i, \mu_i)$ is $\alpha$-mixing,
\item $0 \leq h(T_0) < h(T_1) < \cdots < h(T_k) < \infty$, where $h(T_i)$ denotes the entropy of $T_i$.
\end{enumerate} 
 Suppose that $T_0$ is mixing. Then for any probability measure $\nu$, which is equivalent to $\lambda$, $T_0, T_1, \dots, T_k$ are jointly mixing: for any $B, A_0, \dots, A_k \in \mathcal{B}$,
\begin{equation*}
 \lim_{n \rightarrow \infty} \nu ( B \cap T_0^{-n} A_0 \cap T_1^{-n}A_1 \cap \cdots \cap T_k^{-n}A_k) =   \nu (B) \prod_{i=0}^k \mu_i(A_i).
\end{equation*}
\end{Theorem}

The following theorem is an immediate corollary of Theorem \ref{Int:Thm:JM2:new}.
\begin{Theorem}
Let $\beta_1 , \cdots , \beta_s$ be distinct real numbers with $\beta_i >1$. We assume that $\log \beta_i \ne \frac{\pi^2}{6 \log 2}$ for all $i$. Then for any measurable sets $ B, A_0, A_1 \dots, A_s$, 
\begin{equation*}
\lim\limits_{n \rightarrow \infty} \nu (B \cap T_G^{-n} A_0 \cap T_{\beta_1}^{-n} A_1 \cap \cdots \cap T_{\beta_s}^{-n} A_s) = \nu (B) \, \mu_G (A_0) \, \prod_{i=1}^s \mu_{\beta_i} (A_i), 
\end{equation*}
where $\nu$ is any probability measure equivalent to the Lebesgue measure $\lambda$.
\end{Theorem}

The structure of this paper is as follows.
In Section \ref{sec:joint ergodicity}, we obtain necessary and sufficient conditions for uniform joint ergodicity, which are later used in Section \ref{sec:jointergodicityPM}.  
In Section \ref{sec:PM}, we provide some background material, which is needed for handling systems belonging to the class $\mathcal{T}$. 
In Section \ref{sec:examples}, we present numerous examples of transformations belonging to the class $\mathcal{T}$.
Section \ref{sec:jointergodicityPM} is devoted to the proof of Theorem \ref{Intro:thm:L^2_joint}, which is one of the main results in this paper.  
In Section \ref{sec:jointmixingPM}, we establish useful criteria for joint mixing (and joint weak mixing) for transformations belonging to the class $\mathcal{T}$ and obtain some interesting applications. 
In Section \ref{unifvsnon}, we introduce the non-uniform variant of joint ergodicity and provide some pertinent results and examples.
In Section \ref{sec:disjoint}, we establish additional results on joint ergodicity which utilize the notion of disjointness. 
In Section \ref{sec:final remark}, we discuss possible extensions of the results obtained in this paper, provide some relevant examples, and formulate some natural open questions. Finally, in Section \ref{Sec:App}, we provide a descriptive definition of rank one transformations and show that some of the well-known examples of rank one transformations belong to the class $\mathcal{T}$.

\begin{Remark}
Throughout this paper, we will be tacitly assuming that the probability spaces $(X, \mathcal{B}, \mu)$ that we are working with are regular meaning that $X$ is a compact metric space and $\mathcal{B}$ is the Borel $\sigma$-algebra of $X$. Note that this assumption can be made without loss of generality since every separable measure preserving system is equivalent to a regular one (see for instance, \cite[Proposition 5.3]{Furbook} ).
\end{Remark}

\section{Uniform joint ergodicity}
\label{sec:joint ergodicity}

The main goal of this section is to prove Theorem \ref{Intro:ThmJENS} from the Introduction, which establishes general necessary and sufficient conditions for uniform joint ergodicity. 
For reader's convenience, we formulate this theorem below.
Note that uniform joint ergodicity, which was defined in Definition \ref{Int:Def1.1} above, corresponds to condition (i).
\begin{Theorem}[Theorem \ref{Intro:ThmJENS} in Introduction]
\label{criterion:JE}
Let $(X, \mathcal{B}, \mu)$ be a probability space. Let $T_1, \dots, T_k$ be measurable transformations on $(X, \mathcal{B})$ such that for each $i$, there is a $T_i$-invariant probability measure $\mu_i$, which is equivalent to $\mu$.
The following conditions are equivalent:
\begin{enumerate}[(i)]
\item $T_1, \dots, T_k$ are uniformly jointly ergodic with respect to $(\mu; \mu_1, \dots, \mu_k)$: \\
for any bounded measurable functions $f_1, \dots f_k$, 
\begin{equation}
\label{JE:condition1} 
\lim_{N-M \rightarrow \infty} \frac{1}{N - M} \sum_{n=M}^{N-1} T_1^{n} f_1 \cdot T_2^n f_2 \cdots T_k^n f_k = \prod_{i=1}^k \int f_i \, d \mu_i \quad \text{ in } L^2(\mu). 
\end{equation} 
\item for any bounded measurable functions $f_0, f_1, \dots f_k$, 
\begin{equation}
\label{JE:cond2}
\lim_{N-M \rightarrow \infty} \frac{1}{N - M} \sum_{n=M}^{N-1} \int f_0 \cdot T_1^{n} f_1 \cdot T_2^n f_2 \cdots T_k^n f_k \, d \mu = \int f_0 \, d \mu \cdot \prod_{i=1}^k \int f_i \, d \mu_i. 
\end{equation} 
\item \begin{enumerate}
\item $T_1 \times T_2 \times \cdots \times T_k$ is ergodic with respect to $\mu_1 \times \mu_2 \times \cdots \times \mu_k$
\item for  any bounded measurable functions $f_1, \dots f_k$,
\begin{equation}
\lim_{N-M \rightarrow \infty} \frac{1}{N-M} \sum_{n=M}^{N-1}  \int T_1^{n} f_1 \cdot T_2^n f_2 \cdots T_k^n f_k \, d \mu = \prod_{i=1}^k \int f_i \, d \mu_i .
\end{equation}
\end{enumerate}
\item \begin{enumerate}
\item $T_1 \times T_2 \times \cdots \times T_k$ is ergodic with respect to $\mu_1 \times \mu_2 \times \cdots \times \mu_k$
\item there is $C>0$ such that for any $A_1, A_2, \dots, A_k \in \mathcal{B}$
\begin{equation}
\label{JE:condition2} 
\limsup_{N-M \rightarrow \infty} \frac{1}{N-M} \sum_{n=M}^{N-1} \mu (T_1^{-n} A_1 \cap T_2^{-n} A_2 \cap \cdots \cap T_k^{-n} A_k) \leq C \prod_{i=1}^k \mu_i(A_i).
\end{equation}
\end{enumerate}
\end{enumerate} 
\end{Theorem}

For the proof of the above theorem, we will need the following three useful facts.
The first one is a lemma, which is well known to aficionados. We include the proof for the convenience of the reader.
\begin{Lemma}[cf. Theorem 1 in \cite{PoPy}]
\label{equivalent_measure}
Let $(X, \mathcal{B}, \mu, T)$ be an ergodic measure preserving system. Suppose that $\nu$ is $T$-invariant probability measure on $(X, \mathcal{B})$ with $\nu \ll \mu$. Then $\nu = \mu$.
\end{Lemma}

\begin{proof}
Let $A \in \mathcal{B}$. Note that 
\begin{align*}
\nu(A) &= \nu(T^{-n} A) = \frac{1}{N} \sum_{n=1}^N \nu(T^{-n}A) \\
           &= \frac{1}{N} \sum_{n=1}^N \int 1_A (T^n x) \, d \nu (x)
           = \int  \left( \frac{1}{N}  \sum_{n=1}^N 1_A (T^n x) \right) \, \frac{d \nu}{ d \mu} (x) \, d \mu (x)
\end{align*}
where $\frac{d \nu}{ d \mu} \in L^1(\mu)$ by Radon-Nikodym theorem.
Since $T$ is ergodic with respect to $\mu$, one has $\frac{1}{N}  \sum_{n=1}^N 1_A (T^n x) \rightarrow \mu(A)$ for a.e. $x$. Then we apply dominated convergence theorem to the sequence of functions $ \left( \frac{1}{N}  \sum_{n=1}^N 1_A (T^n x) \right) \, \frac{d \nu}{ d \mu}(x)$ to derive that
\begin{align*} 
\nu(A) &= \lim_{N \rightarrow \infty}  \int  \left( \frac{1}{N}  \sum_{n=1}^N 1_A (T^n x) \right) \, \frac{d \nu}{ d \mu}(x) \, d \mu(x) \\
&= \int \mu(A) \, \frac{d \nu}{ d \mu}(x) \, d \mu(x) = \mu(A) \int \frac{d \nu}{ d \mu}(x) \, d \mu(x)  = \mu(A). 
\end{align*}
\end{proof}

We also will use the following uniform version of the van der Corput lemma. (See Lemma on p.446 in \cite{BeLe}.)
\begin{Lemma}[van der Corput trick] \label{vdC}
Let $(u_n)_{n \in \mathbb{N}}$ be a bounded sequence in a Hilbert space $\mathcal{H}$. 
Then 
\[
\limsup_{N - M \rightarrow \infty} \left\| \frac{1}{N - M} \sum_{n=M}^{N-1} u_n \right\| ^2 \leq  \limsup_{H \rightarrow \infty} \frac{1}{H} \sum_{h=1}^H \limsup_{N - M \rightarrow \infty} \frac{1}{N -M } \sum_{n=M}^{N-1} \mathrm{Re} \langle u_{n+h}, u_n \rangle .
\]
\end{Lemma}

Finally, we will need the following lemma, which provides a convenient characterization of ergodicity of products of measure preserving transformations.
\begin{Lemma}[Proposition 2.1 in \cite{BeBe1}]
\label{prop2.1}
Let $(X_i, \mathcal{B}_i, \mu_i, T_i), 1 \leq i \leq k,$ be measure preserving systems. Their product is ergodic if and only if the following conditions are satisfied:
\begin{enumerate}[(1)]
\item Each $T_i$ is ergodic.
\item  If $\lambda_i$ is an eigenvalue of $T_i$ with $\prod_{i=1}^k \lambda_i =1$, then $\lambda_i =1$ for all $i= 1, 2, \dots, k$. (In case $T_i$ is weakly mixing, we put $\lambda_i =1$.)
\end{enumerate} 
\end{Lemma}

\begin{proof}[Proof of Theorem \ref{criterion:JE}]
It is obvious that $(i) \Rightarrow (ii) \Rightarrow (iii)(b) \Rightarrow (iv)(b)$. 

Let us prove $(ii) \Rightarrow (iii)(a)$.
First, we will show that  the condition $(ii)$ implies that each $T_i$ is ergodic. Indeed, if this is not the case, then, for some $i \in \{1, 2, \dots, k\}$, there is a $T_i$-invariant set $A$ with $0 < \mu(A) < 1$. Let $f_0 = 1_{A^c}$, $f_i = 1_A$ and $f_j =1$ for $j \geq 1, j \ne i$. 
Then substituting into formula \eqref{JE:cond2} gives $\mu(A) \mu(A^c) = 0$, which is a contradiction.  
In light of condition (2) in Lemma \ref{prop2.1}, in order to finish this part of proof, it is sufficient to show that if $\lambda_i$ is an eigenvalue of $T_i$ with $\prod_{i=1}^k \lambda_i =1$, then $\lambda_i =1$ for all $i= 1, 2, \dots, k$. Assume that $T_i f_i = \lambda_i f_i$ for $i \in \{1, 2, \dots, k\}$. (In case $T_i$ is weak mixing, we put $f_i = 1$ and $\lambda_i =1$.) Since each $T_i$ is ergodic, we can assume that $|f_i| =1$. Suppose that $\lambda_j \ne 1$ for some $j$. Also, note that, by ergodicity, if $\lambda_j \ne 1$, then $\int f_j d \mu_j = 0$. If we let $f_0 := \overline{\prod_i f_i}$, then substituting $f_0, f_1, \dots, f_k$ into formula \eqref{JE:cond2} gives $1 = 0$, a contradiction. 

It remains to show $(iv) \Rightarrow (i)$. Let us first show that for any $f_1, \dots, f_k \in L^{\infty}$,
\begin{equation}
\label{refinement_condition(ii)(b)}
\lim_{N-M \rightarrow \infty} \frac{1}{N-M} \sum_{n=M}^{N-1} \int T_1^n f_1 \cdots T_k^n f_k \, d \mu = \prod_{i=1}^k \int f_i \, d \mu_i.
\end{equation}
Let $\mathcal{M}$ be the space of probability Borel measures on $X^k$.
Define a doubly indexed sequence $\nu_{M,N} \in \mathcal{M}, \, M < N,$ by the formula
\[ \nu_{M, N} (A_1 \times A_2 \times \cdots \times A_k) 
:= \frac{1}{N - M} \sum_{n=M}^{N-1} \mu \left(\bigcap_{i=1}^k T_i^{-n} A_i \right), \]
where $A_1, \dots, A_k \in \mathcal{B}$.
Let $\nu$ be any weak* limit point of $\nu_{M,N}$ (as $N-M \rightarrow \infty$). 
It is obvious that $\nu$ is a probability measure, which is invariant under $ T_1 \times T_2 \times \cdots \times T_k$.
Also, condition ($iv$)(b) implies that $\nu$ is absolutely continuous with respect to $ \mu_1 \times \mu_2 \times \cdots \times \mu_k$.
Thus, Lemma \ref{equivalent_measure} implies that  $\nu = \mu_1 \times \mu_2 \times \cdots \times \mu_k$ and so formula \eqref{refinement_condition(ii)(b)} holds for any bounded measurable functions $f_1, \dots, f_k$. 

To prove \eqref{JE:condition1}, in view of the identity $f_1 = (f_1 - \int f_1 \, d \mu_1) + \int f_1 \, d \mu_1 $, it is enough to show that for any real-valued functions $f_1, \dots, f_k \in L^{\infty}$  with $\int f_1 \, d \mu_1 = 0$,
\begin{equation}
\label{eq2.6}
    \lim_{N-M \rightarrow \infty} \frac{1}{N - M} \sum_{n=M}^{N-1} T_1^{n} f_1 \cdot T_2^n f_2 \cdots T_k^n f_k = 0. 
\end{equation}
To show \eqref{eq2.6}, we apply van der Corput trick to $u_n = T_1^n f_1 T_2^n f_2 \cdots T_k^n f_k$.
Set $F(x_1, \dots, x_k) := \prod_{i=1}^k f_i(x_i)$. We have 
\begin{align*}
&\lim_{H \rightarrow \infty} \frac{1}{H} \sum_{h=1}^H \, \lim_{N-M \rightarrow \infty} \frac{1}{N-M} \sum_{n=M}^{N-1} \langle u_{n+h}, u_n \rangle_{L^2(\mu)} \\
&\quad= \lim_{H \rightarrow \infty} \frac{1}{H} \sum_{h=1}^H \, \lim_{N-M \rightarrow \infty} \frac{1}{N-M} \sum_{n=M}^{N-1}  \int T_1^n(f_1 T_1^h f_1) \cdots T_k^n(f_kT_k^h f_k) \, d \mu \\
&\quad=\lim_{H \rightarrow \infty} \frac{1}{H} \sum_{h=1}^H  \prod_{i=1}^k \int f_i T_i^h f_i \, d \mu_i  \quad (\text{due to } \eqref{refinement_condition(ii)(b)}) \\
&\quad = \lim_{H \rightarrow \infty}  \frac{1}{H} \sum_{h=1}^H \int  F(x_1, \dots, x_k) \, ( T_1 \times T_2 \times \cdots \times T_k)^h  F(x_1, \dots, x_k) \, d(\mu_1 \times \mu_2 \times \cdots \times \mu_k) \\
&\quad = \left( \int F(x_1, \dots, x_k)  \, d(\mu_1 \times \mu_2 \times \cdots \times \mu_k) \right)^2  = \prod_{i=1}^k \left( \int f_i \, d \mu_i \right)^2 = 0.
\end{align*}
\end{proof}

The following result, which was obtained in \cite{BeBe2}, is an immediate consequence of the previous theorem.
\begin{Theorem}
\label{JE:BeBe}
Let $T_1, \dots, T_k$ be a measure preserving transformations on a probability space $(X, \mathcal{B}, \mu)$.
The following conditions are equivalent:
\begin{enumerate}[(i)]
\item $T_1, \dots, T_k$ are uniformly jointly ergodic: for any $f_1, \dots f_k \in L^{\infty} (\mu)$, 
\[
\lim_{N -M \rightarrow \infty} \frac{1}{N -M} \sum_{n=M}^{N-1} T_1^{n} f_1 \cdot T_2^n f_2 \cdots T_k^n f_k = \prod_{i=1}^k \int f_i \, d \mu \quad \text{ in } L^2(\mu). 
\]
\item \begin{enumerate}[(a)]
\item $T_1 \times T_2 \times \cdots \times T_k$ is ergodic with respect to $\mu \times \mu \times \cdots \times \mu$
\item for any $f_1, \dots, f_k \in L^{\infty}(\mu)$, 
$$\lim\limits_{N -M \rightarrow \infty} \frac{1}{N -M} \sum\limits_{n=M}^{N-1} \int T_1^{n} f_1 \cdot T_2^n f_2 \cdots T_k^n f_k \, d \mu= \prod_{i=1}^k \int f_i \, d \mu.$$
\end{enumerate}
\end{enumerate} 
\end{Theorem}

The following corollary of Theorem \ref{criterion:JE} will be instrumental for dealing with joint ergodicity of transformations from class $\mathcal{T}$. See Theorem \ref{thm:L^2_joint} in Section \ref{sec:jointergodicityPM}.

\begin{Corollary}
\label{Cor:L2_joint}
Let $(X, \mathcal{B}, \mu)$ be a probability space. Let $\mu_1, \dots, \mu_k$ be probability measures such that $\mu_i$ is equivalent to $\mu$. Let $T_i$ be a measure preserving transformation on $(X, \mathcal{B}, \mu_i)$ for $1 \leq i \leq k$.
Suppose that 
\begin{enumerate}[(1)]
\item $T_1 \times T_2 \times \cdots \times T_k$ is ergodic with respect to $\mu_1 \times \mu_2 \times \cdots \times \mu_k$,
\item there is $C>0$ such that for all $A_1, A_2, \dots, A_k \in \mathcal{B}$
\begin{equation}
\label{2ndcondition} 
\limsup_{n \rightarrow \infty} \mu (T_1^{-n} A_1 \cap T_2^{-n} A_2 \cap \cdots \cap T_k^{-n} A_k) \leq C \prod_{i=1}^k \mu_i(A_i) .
\end{equation}
\end{enumerate} 
Then, $T_1, \dots, T_k$ are uniformly jointly ergodic with respect to $\mu$. 
\end{Corollary}

\section{Some Preparatory Material}
\label{sec:PM}

In this section, we review some useful facts about systems $(T, \mu, \mathcal{A})$ belonging to the family $\mathcal{T}$. 
First, we formulate the entropy equipartition property. 
Then we discuss two technical conditions - property B and $\alpha$-mixing, which were introduced in Section \ref{sec:Intro}. 
The material reviewed in this section will be instrumental for providing the proofs of Theorem \ref{Intro:thm:L^2_joint} (in Section \ref{sec:jointergodicityPM}) and Theorem \ref{Int:Thm:JM2:new} (in Section \ref{sec:jointmixingPM}).

Let $X=[0,1]$ and let $\lambda$ be the Lebesgue measure on $(X, \mathcal{B})$, where $\mathcal{B}$ is the Borel $\sigma$-algebra.
By a partition of the interval $[0,1]$, we mean a finite or countable collection $\mathcal{A} = \{A_i\}$ of intervals such that $\text{int} (A_i) \ne \emptyset$, $\text{int} (A_i) \cap \text{int} (A_j) = \emptyset$ for $i \ne j$ and $[0,1] = \bigcup_i A_i \, \bmod \, 1$. Intervals may be open, closed, and half open on either end. 

\begin{Definition}
Let $(X, \mathcal{B}, \mu, T)$ be an ergodic measure preserving system such that $\frac{1}{c} \leq \frac{ d \lambda}{d \mu} \leq c$ for some $c \geq 1$ and let $\mathcal{A}$ be a partition of $[0,1]$. We will say that $(T, \mu, \mathcal{A}) \in \mathcal{T}$ if 
\begin{enumerate}
\item 
for any interval $I \in \mathcal{A}$, $T|_{\text{int}(I)}$ is continuous and strictly monotone,
\item  partition $\mathcal{A}$ generates the $\sigma$-algebra $\mathcal{B}$, meaning that $\mathcal{B} = \bigvee_{j=0}^{\infty} T^{-j} \sigma( \mathcal{A}) \mod \lambda$, where $\sigma (\mathcal{A})$ denotes the sub-$\sigma$-algebra generated by $\mathcal{A}$,
\item entropy of the partition $\mathcal{A}$ is finite: $H(\mathcal{A}) = - \sum \mu(A_i) \log \mu (A_i) < \infty$. 
\end{enumerate}
\end{Definition}

\subsection{Some basic facts about Kolmogorov-Sinai entropy}
For a measure preserving system $(X, \mathcal{B}, \mu, T)$, $h(T)$ denotes Kolmogorov-Sinai entropy of $T$. 
\subsubsection{Entropy equipartition}
 \begin{Theorem}[Shannon-McMillan-Breiman, see for example Section 6.2 in \cite{Pe}]
\label{entropy:SMB}
Let $(X, \mathcal{B}, \mu, T)$ be a measure preserving system.
Let $\mathcal{A}$ be a  countable partition of $X$ with $- \sum \mu(A_j) \log \mu (A_j) < \infty$ and $  \bigvee_{n=0}^{\infty} T^{-n} \sigma( \mathcal{A} ) = \mathcal{B}$. 
For $x \in X$, let $\mathcal{A}(x)$ be the element of $\mathcal{A}$ to which $x$ belongs.
Denote the information function associated to $\mathcal{A}$ by  $I_{\mathcal{A}} (x) = - \log \mu(\mathcal{A}(x))$. Then
\[ \frac{1}{n} I_{\mathcal{A}^n} (x) \rightarrow h(T) \,\, \text{a.e. and in } L^1. \]
\end{Theorem}

Given $(T, \mu, \mathcal{A}) \in \mathcal{T}$, we say that an interval $A \in \mathcal{A}^n :=  \bigvee\limits_0^{n-1} T^{-j} \mathcal{A}$, $n \in \mathbb{N}$, is a cylinder of rank $n$. The following corollary of Shannon-McMillan-Breiman theorem, which is often referred to as the entropy equipartition property, states roughly that if a cylinder $A$ of rank $n$ belongs to the class of good atoms, then $\mu(A)$ is approximately $e^{-  n h(T)}$. 
\begin{Corollary}
\label{entropy:intervals}
Given $\epsilon > 0$, there is $n_0$ such that for $n \geq n_0$ the elements of $\mathcal{A}^n$ can be divided into two classes 
$\mathfrak{g}^n = \mathfrak{g}^{n} (\epsilon)$ and $\mathfrak{b}^n= \mathfrak{b}^n (\epsilon)$ (classes of ``good" and ``bad" atoms) such that 
\begin{enumerate}
\item $\mu \left( \bigcup_{B \in \mathfrak{b}^n} B \right) < \epsilon$
\item For $A \in \mathfrak{g}^n$,
\begin{equation*}
e^{- n (h(T) + \epsilon)} < \mu(A) < e^{-n (h(T) - \epsilon)}.
\end{equation*}
\end{enumerate}
\end{Corollary}

\begin{Remark}
Note that  for $(T, \mu, \mathcal{A}) \in \mathcal{T}$, $\frac{1}{C} \mu \leq \lambda \leq  C \mu$ for some $C \geq 1$, so
\[ \frac{\log \mu(\mathcal{A}^n(x)) - \log C}{n} \leq \frac{\log \lambda (\mathcal{A}^n(x))}{n} \leq \frac{\log \mu(\mathcal{A}^n(x)) + \log C}{n}. \]
So if a cylinder $A$ of rank $n$ belongs to the class of good atoms, $\lambda(A)$ is approximately $e^{-  n h(T)}$.
\end{Remark} 

\subsubsection{Rokhlin formula}
The following theorem, which goes back to \cite{Rok1}, is useful for computing the entropy $h(T)$ of $(T, \mu, \mathcal{A}) \in \mathcal{T}$.  
\begin{Theorem}[cf. Theorem 2 in \cite{Bu}]
\label{Rokhlin}
Let $T$ be a measure preserving transformation on a probability space $(X, \mu)$. Assume that there is a generating countable measurable partition of $X$ 
into pieces $X_i$ on each of which the restriction $T|X_i$ is one-to-one and non-singular with respect to $\mu$ and $- \sum \mu(X_i) \log \mu (X_i) < \infty$. Let $JT = \frac{d \mu \circ (T|X_i)}{d \mu}$ if $x \in X_i$. Then 
\[ h(T) = \int \log JT d \mu.\]  
\end{Theorem}

\subsection{Property B and $\alpha$-mixing }
\label{sec:propOR}
In this subsection, we provide more details on property B and $ \alpha $-mixing, which were introduced in Section \ref{sec:Intro}. We repeat the definition for convenience of the reader.

\begin{Definition}
Let $(T, \mu, \mathcal{A}) \in \mathcal{T}$.
\begin{enumerate}
\item $(T, \mu, \mathcal{A})$ is said to have {\em property B} if there exist a constant $c > 0$ and a sequence $(b_n)$ of positive real numbers with $b_n \rightarrow 0$ such that  for any natural number $l$ and for any non-negative integer $n$,  if $A \in \sigma(\mathcal{A}^l), B \in \mathcal{B}$,
\[ \mu(A \cap T^{-(n+l)} B) \leq c \mu(A) \mu(B) + b_n.\]
\item $(T, \mu, \mathcal{A})$ is {\em $\alpha$-mixing} if there exists a sequence $\alpha(n), n=1, 2, \dots,$ of non-negative real numbers with $\lim\limits_{n \rightarrow \infty} \alpha(n) = 0$  such that for any natural number $l$, for any non-negative integer $n$, if $A \in \sigma \left( \bigvee_{j=0}^{l-1} T^{-j} \mathcal{A} \right)$ and  $B \in \mathcal{B}$,
\begin{equation} 
\label{eq3.2}
| \mu (A \cap T^{-(n+l)} B) - \mu(A) \mu(B)| \leq \alpha(n).
\end{equation}
\end{enumerate}
\end{Definition}

\begin{Remark}
Since $\frac{1}{C} \mu \leq \lambda \leq C \mu$ for some $C$, property {\em B} can be equivalently formulated as follows: there is $c_0 > 0$ such that for any $n$,  
if $A \in  \sigma (\mathcal{A}^l), B \in \mathcal{B}$,
\[ \lambda(A \cap T^{-(n+l)} B) \leq c_0 \lambda(A) \mu(B) + b_n.\]
\end{Remark}

Recall that a (non-invertible) measure preserving transformation $T$ on a probability space $(X, \mathcal{B}, \mu)$ is called an {\em exact endomorphism} if 
\[ \bigcap_{n=0}^{\infty} T^{-n} \mathcal{B} = \{ \emptyset, X\} \,\, (\bmod \, \mu).\]

Rokhlin (\cite{Rok}) showed that exact endomorphisms have completely positive entropy (meaning that each of its non-trivial factors has positive entropy). This in turn implies that any exact endomorphism is mixing of all orders.

\begin{Theorem}
\label{exactness:new}
If $(T, \mu, \mathcal{A}) \in \mathcal{T}$ is $\alpha$-mixing, then $T$ is exact (and hence has positive entropy).
\end{Theorem}
\begin{proof}
Suppose that $T$ is not exact. 
Then there is a measurable set $B \in \bigcap_{n=0}^{\infty} T^{-n} \mathcal{B}$ with $0 < \mu(B) < 1$, so $0< \mu(B^c)< 1$.  
Since $\mathcal{B} = \bigvee_{i=0}^{\infty} \sigma (T^{-i} \mathcal{A}) \, \bmod \mu$, one can find a sequence of measurable sets $B_n \in \bigvee_{i=0}^{n-1} \sigma (T^{-i} \mathcal{A})$ such that $\mu(B_n \triangle B) \rightarrow 0$. 
Then  $\lambda (B_n) > \epsilon_0$ for some $\epsilon_0 > 0$ if $n$ is sufficiently large.  
Since $B \in \bigcap_{n=0}^{\infty} T^{-n} \mathcal{B}$, one can also find a sequence of measurable sets $C_n \in \mathcal{B}$ such that $T^{-2n} C_n = B^c$. 
Note that 
\[ \mu (B_n \cap T^{-2n} C_n) = \mu (B_n \cap B^c) \leq \mu (B_n \triangle B) \rightarrow 0.\]
By \eqref{eq3.2}, one has
\[ - \alpha(n) + \mu (B_n \cap T^{-2n} C_n) \leq \mu(B_n) \mu(C_n) \leq \alpha(n) + \mu(B_n \cap T^{-2n} C_n).\]
Since $\liminf \mu(B_n) \geq \epsilon_0$, $\mu(C_n) \rightarrow 0$. So,
\[  \mu(B^c) = \lim_{n \rightarrow \infty} \mu (T^{-2n}C_n) =  \lim_{n \rightarrow \infty} \mu (C_n) = 0,\]
which contradicts the assumption that $\mu(B^c)$ is positive.
\end{proof}

We conclude this subsection by defining the family of RU maps, which provide many useful examples to be discussed in forthcoming sections. RU maps form a subfamily of uniformly expanding maps, which will be defined first.  

\begin{Definition}
\label{picewiseexpanding}
Let $\mathcal{A}$ be a partition of the interval $[0,1]$.
A transformation $T: [0,1] \rightarrow [0,1]$ is {\em uniformly expanding} (or piecewise expanding) on the partition $\mathcal{A}$ if 
for each $A  \in \mathcal{A}$ with $\overline{A} = [a,b]$, where $\overline{A}$ denotes the closure of $A$, $T|_{[a,b]} \in C^1([a,b])$\footnote{ It is assumed that at the end points of $a$ and $b$, $T'$ is defined via one-sided derivatives.} and  there is $\alpha > 1$ such that for any $A \in \mathcal{A}$
\[ |T'(x)| \geq \alpha \text{ for } x \in [a,b].\]
\end{Definition}

Now we introduce the notion of RU maps.  
Define variation of $f$ on $I= [0,1]$ by 
\[ \text{Var} (f) := \sup \sum_i |f(x_i) - f(x_{i-1})|,\]
where the supremum ranges over all $0 = x_0 < x_1 < \cdots < x_n =1$. 
For every $f \in L^1(\lambda)$, set
$$\| f \|_{BV} := \| f \|_{\infty} + V(f), \, \text{where} \, V(f) = \inf \{\text{Var} (f^*): f^* = f \, \,\, \text{a.e. with respect to } \lambda\}. $$
Let $BV = :\{f \in L^1(\lambda): \|f\|_{BV} < \infty  \}$.
Then the space $BV$ is a Banach space with the norm $\|  \cdot \|_{BV}$.

\begin{Definition}[cf. \cite{ADSZ},  p.19 and \cite{AaNa}, p.3] 
\label{def:RU}
A map $T:[0,1] \rightarrow [0,1]$ is said to be a {\em RU map}
if there exist a partition $\mathcal{A}$ such that
\begin{enumerate}
    \item  $T$ is uniformly expanding (on the partition $\mathcal{A}$) 
    \item $T$ satisfies {\em Rychlik's condition}:
\[ \sum_{A \in \mathcal{A}} \| 1_{TA} v_A' \|_{BV} < \infty,\] 
where for $A \in \mathcal{A}$, $v_{A}$ denotes the inverse of $T: A \rightarrow TA$.
\end{enumerate}
\end{Definition}

The following examples illustrate Definition \ref{def:RU}. (More examples of RU maps will be given in Section \ref{sec:examples}.)
\begin{ex}
\label{ex3.12}
\mbox{}
\begin{enumerate}
\item $Tx = 2x \, \bmod \, 1$. The natural partition $\mathcal{A}$ consists of $I_0 = [0, 1/2)$ and $I_1 =[1/2,1)$. Then $T' (x) = 2$ on each $I_j$ and $1_{TI_j} = [0,1)$, $j = 0, 1$. So we have that
$ \sum_{I_j \in \mathcal{A}} \| 1_{TI_j} v_{I_j}' \|_{BV} = 1$.

\item $\beta$-transformations. Let $T_{\beta}x = \beta x \, \bmod \, 1$, where $\beta > 1$ not an integer.
Take the natural partition $\mathcal{A}$ consisting of intervals $ I_j = [\frac{j}{\beta}, \frac{j+1}{\beta}), j = 0, 1, \dots, [\beta]-1$ and  
$I_{[\beta]} = [ \frac{[\beta]}{\beta}, 1)$. 
For each $I_j$, $T'(x) = \beta$ and $TI_{j} = [0,1)$ for $j = 0, 1, \dots, [\beta]-1$ and $TI_{[\beta]} = [0, \beta - [\beta] )$. Thus, $ \sum_{I_j \in \mathcal{A}} \| 1_{TI_j} v_{I_j}' \|_{BV} = \frac{[\beta]+1}{\beta}$.

\item L{\"u}roth transformation. 
$$ T(x) = 
\begin{cases} 
      n(n+1)x - n, & \text{ if }  \frac{1}{n+1} \leq x < \frac{1}{n},    \\
      0, & \text{ if } x = 0, 1.
   \end{cases}
$$ 
Take the natural partition $\mathcal{A}$ consisting of intervals $ I_n = [\frac{1}{n+1}, \frac{1}{n}), n \in \mathbb{N}$.  If $x \in I_n$, $Tx = n(n+1)x - n$, so for each $I_n$, $T'(x) = n(n+1)$ and $TI_n = [0,1)$. Thus, $ \sum_{I_n \in \mathcal{A}} \| 1_{TI_n} v_{I_n}' \|_{BV} = \sum \frac{1}{n(n+1)} = 1$.
\end{enumerate}
\end{ex}

Aaronson and Nakada (\cite{AaNa}) showed that  if $(T, \mu, \mathcal{A}) \in \mathcal{T}$ is a weakly mixing RU map, then it is $\alpha$-mixing (in \cite{AaNa},  $\alpha$-mixing is called $\phi_{-}$-mixing):
\begin{Theorem}[Theorem 1 in \cite{AaNa}]
\label{Thm:AaNa}
Let $(T, \mu, \mathcal{A}) \in \mathcal{T}$. If $T$ is a weakly mixing RU map, then there is a real number $r$ with $0 < r <1$ such that for any $A \in \sigma (\mathcal{A}^l)$ and $B \in \mathcal{B}$, 
\begin{equation*}
|\mu(A \cap T^{-n-l} B) - \mu(A) \mu (B)| \leq  \mu(B) O(r^n),
\end{equation*}
and so $(T, \mu, \mathcal{A})$ is $\alpha$-mixing.
\end{Theorem}

\section{Examples of transformations}
\label{sec:examples}

In this section we collect various examples of transformations belonging to $\mathcal{T}$. These examples will be used in the subsequent sections to illustrate our results pertaining to uniform joint ergodicity and joint mixing.  
 
\subsection{Transformations in $ \mathcal{T}$ which are $ \alpha$-mixing.}
\label{alphamixingexample} 

\subsubsection{Linear $ \bmod \, 1$ transformations.} 
\mbox{}

Times $b$ maps $T_b$ and $\beta$-transformations $T_{\beta}$, which were mentioned in the introduction, belong to a general family of maps $T_{\beta,\gamma}$:
$$T_{\beta, \gamma} x = \beta x + \gamma \bmod 1, \,\, (\beta >1, 0 \leq \gamma<1).$$
It is known that if either $\beta \geq 2, 0 \leq \gamma <1$ or $1 < \beta < 2, \gamma = 0$, then there is a $T_{\beta, \gamma}$-invariant and ergodic measure  $\mu_{\beta, \gamma}$ such that $1/c \leq \frac{d \mu_{\beta, \gamma}}{ d \lambda} \leq c$ for some $c \geq 1$. (See \cite{Hal} and \cite{Wil}.) On the other hand, Parry showed that if $1 < \beta <2$ and $\gamma \ne 0$, then  $T_{\beta, \gamma}$ may not have  an invariant measure, which is equivalent to the Lebesgue measure. (See \cite{Pa2}).
In this paper, we always assume that $\beta \geq 2, 0 \leq \gamma <1$ or $1 < \beta < 2, \gamma = 0$. 

To see that $T_{\beta, \gamma}$ belongs to $\mathcal{T}$, consider the following natural partition of the interval $[0,1]$: 
\[ \mathcal{A}_{\beta, \gamma} = \left\{ \left[ 0, \frac{1 - \gamma}{\beta} \right),  \left[ \frac{1 - \gamma}{\beta},  \frac{2 - \gamma}{\beta} \right) \cdots, \left[ \frac{[\beta + \gamma] -1 - \gamma}{\beta}, \frac{[\beta + \gamma] - \gamma}{\beta} \right),  \left[ \frac{[\beta + \gamma] - \gamma}{\beta}, 1 \right)  \right\}. \]
If $\beta = b \geq 2$ is an integer and $\gamma =0$, then the last element $\left[ \frac{[\beta + \gamma] - \gamma}{\beta}, 1 \right)$ in $\mathcal{A}_{\beta, \gamma}$ can be omitted, and in this case we will deal with the following partition 
\begin{equation*}
\mathcal{A}_b = \left\{ \left[0, \frac{1}{b} \right), \cdots, \left[ \frac{b-1}{b}, 1 \right)  \right\}. 
\end{equation*}
 Rokhlin formula (Theorem \ref{Rokhlin}) implies that entropy of $T_{\beta, \gamma}$ is given by $h(T_{\beta, \gamma}) = \log \beta$. 
Also $T_{\beta, \gamma}$ is a weakly mixing RU map (with respect to $\mathcal{A}_{\beta, \gamma}$ and $\mu_{\beta, \gamma}$), so it is $\alpha$-mixing by Theorem \ref{Thm:AaNa}.

\subsubsection{Piecewise $C^2$-expanding maps}
\mbox{}

We consider a piecewise map $T: [0,1] \rightarrow [0,1]$ with the following properties: there is a finite partition $\mathcal{A} = \{ I_1, \dots, I_N\}$ of $[0,1]$, where $I_i = [a_i, a_{i+1}]$ and $0 = a_0 < a_1 < \cdots < a_{N+1} =1$ such that
\begin{enumerate}[(i)]
\item (smoothness) for each $i \in \{ 1, 2, \dots, N \}$, $T|_{\text{int }(I_i)}$ has $C^2$-extension to $I_i$,
\item (local invertibility) for each $i \in \{ 1, 2, \dots, N \}$, $T$ is strictly monotone on $I_i$,
\item (Markov property) for each $J \in \mathcal{A}$, there is a subset of intervals $P(J) \subset \mathcal{A}$ such that $$\overline{T(\text{int }(J))} = \bigcup \{ K : K \in P(J) \},$$
\item (aperiodicity) for each $J \in \mathcal{A}$, there is a positive integer $q$ such that $T^q({J}) = [0,1].$
\end{enumerate}

One can check that times $b$ maps $T_b$ and $\beta$-transformation for $\beta = \frac{1+ \sqrt{5}}{2}$ satisfy the above conditions (i) - (iv). Actually, $T_{\beta}$ satisfies condition (iii) only for a rather special class of algebraic numbers (see \cite{Bla}, Proposition 4.1).

The following (so called folklore theorem) is a convenient source for many examples of transformations belonging to $\mathcal{T}$.
\begin{Theorem}[Theorem 6.1.1 in \cite{BoGo}]
\label{folklore}
If $T$ satisfies (i) - (iv) above and it is uniformly expanding on the partition $\mathcal{A}$,
then there exists an invariant probability measure $\mu$ equivalent to $\lambda$ such that 
$T$ is ergodic (actually exact) and 
\[ \frac{1}{c} \leq \frac{d \mu}{ d \lambda}  \leq c\]
for some $c \geq 1$.
\end{Theorem}

It is not hard to check that transformations $T$, which are uniformly expanding and satisfy the above conditions (i) - (iv), are weakly mixing RU maps (see Definition \ref{def:RU}), and hence are $\alpha$-mixing (see Theorem \ref{Thm:AaNa}).

\subsubsection{Restrictions of finite Blaschke products to the unit circle $S^1$.}
\label{blaschke}
\mbox{}

Consider a finite Blaschke product $f$ defined on the closed unit disk $\overline{\mathbb{D}} = \{z : |z| \leq 1\}$ by 
\[ f(z) = C \prod_{j=1}^M \frac{z - a_j}{1- \overline{a_j} z},\]
where $|C| = 1$, $|a_j| < 1$ for $j = 1, 2, \dots, M$ and $M  \geq 2$. 

Let $\tau$ be the restriction of $f$ to the unit circle $S^1 = \{z: |z| =1\}$. Observe that $\tau$ is an $M$-to-$1$ map of $S^1$ to $S^1$. 
Denote a normalized Lebesgue measure on $S^1$ by $\sigma$. 
Martin \cite{Mar} showed that if 
\[ \sum_{j=1}^M \frac{1 - |a_j|}{1+ |a_j|} > 1,\]
then 
\begin{enumerate}
\item there exists a unique fixed point $z_0 \in \mathbb{D}$ for the map $f$,
\item there exists a unique $\tau$-invariant probability measure $\mu_{z_0}$ on $S^1$ with 
\[ d \mu_{z_0} = \frac{1 - |z_0|^2}{|u-z_0|^2} d \sigma, \quad u \in S^1,\]
\item $(\tau, \mu)$ is an exact endomorphism,
\item entropy $h(\tau)$ is given by 
\[ h(\tau) = \int_{S^1} \frac{1 - |z_0|^2}{|u-z_0|^2} \log \sum_{j=1}^M \frac{1 - |a_j|^2}{|u - a_j|^2} \, d \sigma (u).\]
\end{enumerate}

Notice that via a measurable isomorphism $\phi:  [0,1] \rightarrow S^1$ given by $\phi(t) = e^{2 \pi i t}$, the transformation $T = \phi^{-1} \circ \tau \circ \phi$, which is isomorphic to $\tau$, satisfies the conditions in Theorem \ref{folklore}.

\subsubsection{Gauss map} $T_G x = \frac{1}{x} \, \bmod \, 1$.  
As we have already mentioned in the Introduction, this map preserves the Gauss measure $d \mu_G = \frac{1}{\log 2} \frac{dx}{1+x}$ and the natural partition for $T_G$ belong to $\mathcal{T}$ is
\begin{equation*}
\mathcal{A}_G = \left\{  \left( \frac{1}{n+1}, \frac{1}{n} \right]: n \in \mathbb{N} \right\}.
\end{equation*}
It is well-known that  $T_G$ is strongly mixing and the entropy of Gauss map is $h(T_G) = \frac{\pi^2}{6 \log 2}$.
Furthermore $(T_G, \mu_G, \mathcal{A}_G)$ is $\alpha$-mixing:
\begin{Theorem}[\cite{Go}]
 There is a constant $r$ with $0 < r < 1$ such that for $A \in \sigma ((\mathcal{A}_G)^l)$, $B \in \mathcal{B}$ and  for $n \in \mathbb{N}$
\begin{equation*}
\left| \mu_G(A \cap T_G^{-n-l} B) - \mu_G(A) \mu_G(B)  \right| =   \mu_G(A) \mu_G(B) O(r^n).
\end{equation*}
\end{Theorem}

\subsubsection{L{\"u}roth series transformations} 
\mbox{}

The map 
$$ T(x) = 
\begin{cases} 
      n(n+1)x - n, & \text{ if }  \frac{1}{n+1} \leq x < \frac{1}{n},    \\
      0, & \text{ if } x = 0, 1.
   \end{cases}
$$  
is called the L{\"u}roth transformation and appears in the study of the L{\"u}roth expansion: every $x \in [0,1)$ can be written as a finite or infinite series
\begin{equation}
\label{Lexp} 
x = \frac{1}{a_1} + \frac{1}{a_1 (a_1 -1) a_2} + \cdots + \frac{1}{a_1 (a_1 -1) \cdots a_{n-1} (a_{n-1} -1) a_n} + \cdots, 
\end{equation}
where $a_n \geq 2$ for each $n \geq 1$. 
L{\"u}roth (\cite{Lu}) showed that each irrational number can be represented as a unique infinite expansion in \eqref{Lexp} and each rational number can be represented as either a finite or a periodic expansion. 

Generalized L{\"u}roth series (GLS) transformations are defined in the following way. (See \cite{BaBuDaKr} for more information.)
Consider any partition $\mathcal{A} = \{ [l_n, r_n) : n \in D \}$ of $[0,1]$, where $D \subset \mathbb{N}$ is finite or countable and $\sum\limits_{n \in D} (r_n - l_n) =1$. 
Write $L_n = r_n - l_n$  for $n \in D$ and assume that $0 < L_i \leq L_j < 1$ for $i > j$. 
Define 
\[ T_D(x) = \frac{x - l_n}{ r_n - l_n}, \, x \in [l_n, r_n), \quad S_D(x) = \frac{r_n - x}{r_n -l_n}, \, x \in [l_n, r_n). \]
($T_D$ and $S_D$ are defined for $\lambda$-almost every $x$.)
 Let $\epsilon= (\epsilon(n))_{n \in D}$ be an arbitrarily, fixed sequence of zeros and ones. Define GLS transformations $T_{D, \epsilon}: [0,1] \rightarrow [0,1]$ by 
\[ T_{D, \epsilon} (x) : = \epsilon(x) S_D(x) + (1- \epsilon(x)) T_D(x),\]
where  $\epsilon(x) =      \epsilon(n), \,  x \in [l_n, r_n).$
Notice that if $\mathcal{A} = \{ \left[ \frac{1}{n+1}, \frac{1}{n} \right) : n \in \mathbb{N} \}$ and $\epsilon (n) = 1$ for all $n$, then it is the L{\"u}roth transformation.

It is known that  $T_{D, \epsilon}$ is $\lambda$-preserving and strong mixing. Moreover, if $H (\mathcal{A}) := - \sum_{n \in D} L_n \log L_n$ is finite, then $h(T) = H(\mathcal{A})$. One can easily check that $T_{D, \epsilon}$ is a RU map (see Example \ref{ex3.12}), so it is $\alpha$-mixing.

\subsection{Transformations in $\mathcal{T}$ having zero entropy}
All the examples in the previous subsection have positive entropy. 
In this subsection we will list some natural examples of transformations belonging to $\mathcal{T}$ which have zero entropy. 

\subsubsection{Interval exchange transformations}
\label{IET}
\mbox{}

These transformations can be viewed as a generalization of circle rotation $T(x) =  x + \alpha \, \bmod \, 1$, which corresponds to the case of two intervals. Recall the definition from the Introduction: for a permutation $\pi$ of the symbols $\{1, 2, \dots, n\}$ and a vector $\mathbf{a} = (a_1, \dots, a_n)$, where $a_i > 0$ for all $i$ and $\sum a_i = 1$,  
an $(\mathbf{a}, \pi)$-interval exchange transformation is defined by 
\[ T_{\mathbf{a}, \pi} (x) = x -  \sum_{j<i} a_j + \sum_{\pi(j) < \pi(i)} a_{j}, \quad x \in I_i,\]
where $a_0 = 0$ and $I_i = [ a_0 + \cdots + a_{i-1}, a_0 + \cdots a_i )$. 

Interval exchange transformations are Lebesgue measure preserving and have zero entropy (this follows from Rokhlin formula (see Theorem \ref{Rokhlin})). 
Let $\mathcal{A} = \{ \overline{I_i}: i = 1, 2, \dots, n\}$. 
It is clear that if $(T_{\mathbf{a}, \pi})$ is ergodic, then  $(T_{\mathbf{a}, \pi}, \lambda, \mathcal{A}) \in \mathcal{T}$.

 Avila and Forni (\cite{AvFo}) proved that if $\pi$ is an irreducible permutation on $\{1, 2, \dots, n\}$ which is not a rotation,\footnote{A permutation $\pi$ is called {\em irreducible} if $\pi (\{1,2, \dots, k\}) \ne \{1, 2, \dots, k\}$ for any $1 \leq k < n$ and it is called a {\em rotation} if $\pi(i+1)\equiv \pi(i) +1 \, \bmod \, n$, for all $i \in \{1, 2, \dots, n\}$.} then for almost every choice of $\mathbf{a} = (a_1, \dots, a_n)$ in the simplex $a_1 + \dots + a_n =1$, the corresponding interval exchange transformation $T_{\mathbf{a}, \pi}$ is weakly mixing.
Finally, it is worth mentioning that interval exchange transformations cannot be strongly mixing. (See \cite{Kat}.) 

\subsubsection{Rank one transformations on the interval $[0,1]$}
\label{subsec:rank1}
\mbox{}

The family of rank one transformations serves as a convenient source of examples exhibiting variety of dynamical properties.
 While it is not hard to define any particular rank one system, the general definition is quite cumbersome.  
There are actually a number of equivalent definitions, but as Ferenczi says in \cite{Fer}, 
``each of them may be useful in some context, and none of them is short and easy to explain."
Roughly speaking, any rank one transformation $T$ may be realized as a Lebesgue measure preserving transformation on $[0,1]$, which is obtained by a so-called cutting and stacking construction. 

Rank one systems are ergodic and have zero entropy. Ornstein (\cite{Or}) showed that rank one transformations can be strongly mixing. 
Also, it is known that mixing rank one transformations are mixing of all orders (\cite{Fer}, \cite{Kal}, \cite{Ryz1}, \cite{Ryz2}).  
Veech (\cite{Vee}) showed that  if $\pi$ is an irreducible permutation on $\{1, 2, \dots, n\}$, for almost every $\mathbf{a}$, the interval exchange transformation $T_{\mathbf{a}, \pi}$ is rank one. 

Here are some examples of rank one transformations, which (along with the appropriate partitions) belong to the class $\mathcal{T}$. For a detailed definition of rank one transformations and some additional discussion of examples (including the items (2) and (3) below), see the Appendix.
\begin{enumerate}
\item Von Neumann-Kakutani transformation: 
\begin{equation}
\label{form:VK}
    T \left(1 - \frac{1}{2^n} + x \right) = \frac{1}{2^{n+1}} + x, \quad \text{for } 0 \leq x < \frac{1}{2^{n+1}}, \, n = 0, 1, 2, \cdots.
\end{equation}
In other words, $T$ linearly maps $[0, 1/2)$ to $[1/2, 1)$, $[1/2, 3/4)$ to $[1/4,1/2)$ and in general 
$ [ (2^{n-1} -1)/2^{n-1}), (2^n-1)/2^n) \text{ to } [1/2^n, 1/2^{n-1}).$
It is well known that this map is ergodic (but not totally ergodic\footnote{ A measure preserving transformation $T$ on $(X, \mathcal{B}, \mu)$ is called {\em totally ergodic} if $T^k$ is ergodic for all $k \in \mathbb{N}$.}) and has discrete spectrum. 
The eigenvalues of $T$ are of the form $e^{2 \pi i \alpha}$, where $\alpha = \frac{m}{2^n}$ ($m, n \in \mathbb{Z}$). (See \cite{Fri}, Example 6.4, p.82  and \cite{Nad2}, p.40-43.)
\item Chacon transformation (\cite{Cha}). This map is weakly mixing but not strongly mixing.
\item Smorodinsky-Adams transformation (\cite{Ad}). This map is strongly mixing.
\end{enumerate}

\section{Uniform joint ergodicity of transformations in $\mathcal{T}$}
\label{sec:jointergodicityPM}

In this section we prove Theorem \ref{Intro:thm:L^2_joint}, which was stated in the Introduction, and illustrate it with some pertinent examples.  
\begin{Theorem}[Theorem \ref{Intro:thm:L^2_joint} in Introduction]
\label{thm:L^2_joint}
For $i = 0, 1, 2, \dots, k$, let $\mu_i$ be a probability measure on the measurable space $([0,1], \mathcal{B})$ and let  $T_i:[0,1] \rightarrow [0,1]$ be $\mu_i$-preserving ergodic transformations such that
\begin{enumerate}
\item for $i =0, 1, \dots, k$, there is a partition $\mathcal{A}_i$ with  $(T_i, \mu_i, \mathcal{A}_i) \in \mathcal{T}$,
\item for $i = 1, 2, \dots, k$, $(T_i, \mu_i, \mathcal{A}_i)$ satisfies property B,
\item $0 \leq h(T_0) < h(T_1) < \cdots < h(T_k) < \infty$, where $h(T_i)$ denotes the entropy of $T_i$.
\end{enumerate} 
Moreover, suppose that $T_0 \times T_1 \times \cdots \times T_k$ is ergodic on $(X^{k+1}, \otimes_{i=0}^k \mu_i)$. 
Then for any probability measure $\nu$, which is equivalent to the Lebesgue measure $\lambda$,
transformations $T_0, T_1, \dots, T_k$ are uniformly jointly ergodic with respect to $(\nu; \mu_0, \mu_1, \dots, \mu_k)$: 
for any $f_0, f_1, \dots, f_k \in L^{\infty}$, 
\begin{equation}
\label{thm4.1:eq1}
    \lim\limits_{N -M \rightarrow \infty} \frac{1}{N-M } \sum\limits_{n=M}^{N-1}  T_0^{n} f_0 \cdot T_1^n f_1 \cdots T_k^n f_k = \prod_{i=0}^k \int f_i \, d \mu_i \quad \text{ in } L^2(\nu).
\end{equation}
\end{Theorem}

In the proof of Theorem \ref{thm:L^2_joint}, we will use the observation that if $\mu$ and $\nu$ are equivalent probability measures, then $T_1, \dots, T_k$ are uniformly jointly ergodic with respect to  $(\mu; \mu_1, \dots, \mu_k)$ if and only if $T_1, \dots, T_k$ are uniformly jointly ergodic with respect to  $(\nu; \mu_1, \dots, \mu_k)$.
This fact is an immediate consequence of the following simple lemma.

\begin{Lemma}
\label{Lem:Conv:Mean}
Let $\mu$ and $\nu$ be probability measures on a measurable space $(X, \mathcal{B})$ such that $\mu$ and $\nu$ are equivalent. 
If a sequence of functions $F_n$ with $\| F_n \|_{\infty} \leq 1$ converges to $F$ in $L^2(\mu)$, then $F_n$ converges to $F$ in $L^2(\nu)$. 
\end{Lemma}
\begin{proof}
Define $A_{\delta} = \{ x: \frac{d \nu}{d \mu}(x) > \frac{1}{\delta} \}$. Note that $\lim\limits_{\delta \rightarrow 0} \nu(A_{\delta}) =0$. 
Then 
\begin{align*}
\int |F_n(x) - F(x)|^2 \, d \nu(x) & = \int |F_n(x) - F(x)|^2 \, \frac{d \nu}{d \mu}(x) \, d \mu \\
&= \int_{A_{\delta}}  |F_n(x) - F(x)|^2 \, \frac{d \nu}{d \mu}(x) \, d \mu + \int_{A_{\delta}^c} |F_n(x) - F(x)|^2 \, \frac{d \nu}{d \mu}(x) \, d \mu \\
&\leq 4 \int_{A_{\delta}} d \nu + \frac{1}{\delta} \int |F_n(x) - F(x)|^2  \, d \mu.
\end{align*}
Letting $n \rightarrow \infty$ and then $\delta \rightarrow 0$, we can conclude $F_n \rightarrow F$ in $L^2(\nu)$.
\end{proof}

\begin{proof}[Proof of Theorem \ref{thm:L^2_joint}]
In light of Lemma \ref{Lem:Conv:Mean}, we may and will assume that $\nu = \lambda$.
Since $(T_i, \mu_i, \mathcal{A}_i) \in \mathcal{T}$, there exists $c \geq 1$ such that for any $i = 0, 1, \dots, k$, 
\begin{equation}
\label{eq:measurebound} 
\frac{1}{c} \mu_i \leq \lambda \leq c \mu_i.
\end{equation}
Since $(T_i, \mu_i, \mathcal{A}_i)$ has property B for $1 \leq i \leq k$, 
there exist $c_0 \geq 1$ and a sequence $(b_n)$ with $b_n \rightarrow 0$ such that for any $i= 1, 2, \dots, k$ and for any $p \geq 0$, if $A \in  \sigma ((\mathcal{A}_i)^p)$ and $B  \in \mathcal{B}$,
\[ \lambda (A \cap T_i^{-(n+p)} B) \leq c_0 \lambda (A) \mu_i (B) + b_n \,\, \text{ for any } n \geq 1.\] 

By Corollary \ref{Cor:L2_joint}, in order to prove \eqref{thm4.1:eq1}, it is enough to show that for any $A_0, \dots, A_k \in \mathcal{B}$,
 \begin{equation}
\label{pf-je-main}
\limsup_{n \rightarrow \infty} \lambda (T_0^{-n} A_0 \cap T_1^{-n} A_1 \cap \cdots \cap T_k^{-n} A_k) \leq c \cdot c_0^k \prod_{i=0}^k \mu_i (A_i).
\end{equation}
To prove \eqref{pf-je-main} it is enough to show that it holds for any measurable sets $A_0, \dots, A_k$ with $A_i \in (\mathcal{A}_i)^l$ for $i = 0, 1, 2, \dots, k$. (The validity of \eqref{pf-je-main} for general $A_0, \dots, A_k \in \mathcal{B}$ will follow then by the standard approximation argument.)

Write $h_i = h(T_i)$ for simplicity.  Let $\epsilon > 0$ such that $h_i - h_{i-1} - 2 \epsilon > 0$ for $i= 1, 2, \dots, k$. 
Take $\delta, \gamma >0$ such that for all $i = 1, 2, \dots, k$, if $n$ is sufficiently large, one has 
\begin{equation}
\label{eq:gamma}
(h_{i} - \epsilon) (1- \delta)n  - (h_{i-1} + \epsilon)(n+l)  > \gamma n.
\end{equation}

We use the following notation: for $i = 0, 1 , \dots, k$,
\begin{itemize}
\item $\mathfrak{g}_i^n = \{A \in (\mathcal{A}_i)^n:   e^{-n (h_i + \epsilon)} < \mu_i(A) <  e^{- n (h_i - \epsilon)}\}$,
\item $\mathfrak{b}_i^n = (\mathcal{A}_i)^n  \setminus \mathfrak{g}_i^n$. 
\end{itemize}
Due to \eqref{eq:measurebound}, $\frac{1}{c} e^{-n (h_i + \epsilon)} < \lambda(A) < c e^{- n (h_i - \epsilon)}$ for $A \in \mathfrak{g}_i^n$. 
And Corollary \ref{entropy:intervals} says that there is $N_0 = N_0 (\epsilon)$ such that if $n \geq N_0$, then $\mu_i (\bigcup_{B \in  \mathfrak{b}_i^n} B) \leq \epsilon$, 
so $\lambda (\bigcup_{B \in  \mathfrak{b}_i^n} B) \leq c \epsilon$.

In the remaining part, we assume that $n$ is sufficiently large so that (1) \eqref{eq:gamma} holds, (2) $n - [\delta n] \geq N_0$ and (3) $\gamma n \geq 2 \log c$.

First we will estimate $\lambda(T_0^{-n} A_0)$ and then use induction to estimate $ \lambda (T_0^{-n} A_0 \cap T_1^{-n} A_1 \cap \cdots \cap T_k^{-n} A_k)$.

Note that $T_0^{-n} A_0 \in \sigma ( (\mathcal{A}_0)^{n+l})$. 
Corollary \ref{entropy:intervals} says that one can write 
\begin{equation*}
 T_0^{-n} A_0 = I_1(0) \cup I_2 (0),
 \end{equation*}
where 
\begin{itemize}
\item $I_1(0)$ is union of some intervals in $ \mathfrak{g}_0^{n+l}$, more precisely, $I_1(0) = \bigcup_{i=1}^{s(0)} J_i$ such that $J_i \in  \mathfrak{g}_0^{n+l}$
\item  $I_2(0)$ is union of some intervals in $ \mathfrak{b}_0^{n+l}$, so $\lambda (I_2 (0)) \leq c \epsilon$.
\end{itemize}
Moreover, $\lambda (I_1(0)) \leq \lambda (T_0^{-n} A_0) \leq c \mu_0 (T_0^{-n} A_0) = c \mu_0 (A_0)$.  

To estimate $ \lambda (T_0^{-n} A_0 \cap T_1^{-n} A_1 \cap \cdots \cap T_k^{-n} A_k)$, we need the following lemma.
\begin{Lemma}
\label{sec4:lem1}
 Suppose that for $0 \leq q \leq k-1$, there are sets $I_1(q)$ and $I_2(q)$ such that
\[T_0^{-n} A_0 \cap T_1^{-n} A_1 \cap \cdots \cap T_q^{-n} A_q = I_1(q) \cup I_2(q), \]
where 
\begin{enumerate}[(i)]
\item $I_1(q) = \bigcup_{i=1}^{s(q)} J_i(q)$, where $J_i(q) \in \mathfrak{g}_q^{n+l}$ for $1 \leq i \leq s(q)$,
\item $I_2(q)$ is a set with $\lambda(I_2(q)) = O(\epsilon + e^{-\gamma n})$,
\item $\lambda(I_1(q)) \leq c \cdot c_0^{q} \prod_{i=0}^q \mu_i(A_i) + (q-1) b_{[\delta n]}$.
\end{enumerate}
Then, there are sets $I_1(q+1)$ and $I_2(q+1)$ such that 
\begin{equation*}
    T_0^{-n} A_0 \cap T_1^{-n} A_1 \cap \cdots \cap T_{q+1}^{-n} A_{q+1} = I_1(q+1) \cup I_2(q+1),
\end{equation*}

where 
\begin{enumerate}
\item $I_1(q+1) = \bigcup_{i=1}^{s(q+1)} J_i(q+1)$, where $J_i(q+1) \in \mathfrak{g}_{q+1}^{n+l}$ for $1 \leq i \leq s(q+1)$,
\item $I_2(q+1)$ is a set with $\lambda(I_2(q+1)) = O(\epsilon + e^{-\gamma n})$,
\item $\lambda(I_1(q+1)) \leq c \cdot c_0^{q+1} \prod_{i=0}^{q+1} \mu_i(A_i) + q b_{[\delta n]}$.
\end{enumerate}
\end{Lemma}

If Lemma \ref{sec4:lem1} holds, then we have that 
\[ \lambda (T_0^{-n} A_0 \cap T_1^{-n} A_1 \cap \cdots \cap T_k^{-n} A_k) 
\leq c \cdot c_0^k \prod_{j=0}^{k} \mu_j (A_j) + k b_{[\delta n]} + O ( \epsilon + e^{-\gamma n}). \]
Letting $n \rightarrow \infty$, we have 
\[ \limsup_{n \rightarrow \infty} \lambda (T_0^{-n} A_0 \cap T_1^{-n} A_1 \cap \cdots \cap T_k^{-n} A_k) 
\leq c \cdot c_0^k \prod_{j=0}^{k} \mu_j (A_j) + O ( \epsilon). \]
Since one can choose $\epsilon$ arbitrarily small, \eqref{pf-je-main} holds.
Thus it remains to prove Lemma \ref{sec4:lem1}.

Let $m = [\delta n]$ for the sake of simplicity.  
To prove Lemma \ref{sec4:lem1},
we will show that, up to a small error term, one can represent $I_1(q)$ 
as union of intervals in $ \mathfrak{g}_{q+1}^{n-m}$. This is formalized in the following auxiliary lemma.
\begin{Lemma}
\label{sec4:lem2} 
We have
\begin{equation*}
\label{I_1}
 I_1(q) = H_1(q) \cup H_2(q),
 \end{equation*}
where 
\begin{itemize}
\item $H_1(q)$ is union of some intervals in $\mathfrak{g}_{q+1}^{n-m}$,
\item $\lambda (H_2(q)) = O (\epsilon + e^{-\gamma n})$.
\end{itemize}
\end{Lemma}
We will prove now Lemma \ref{sec4:lem2} (and after that we will provide the proof of Lemma \ref{sec4:lem1}). 

\noindent {\em Proof of Lemma \ref{sec4:lem2}.}
Note that  
\begin{equation}
\label{estimate:s}
 s(q) \leq c e^{(n+l) (h_q + \epsilon)}
\end{equation}
since $\lambda(J_i(q)) \geq \frac{1}{c} e^{-(n+l) (h_q + \epsilon)}$, $ \lambda (\bigcup_{i=1}^{s(q)} J_i(q)) \leq 1$ and $J_i(q)$, $1 \leq i \leq s(q)$, are disjoint. 
Also one can see that
\begin{enumerate}
\item $(\mathcal{A}_{q+1})^{n-m} (= \mathfrak{g}_{q+1}^{n-m} \cup \mathfrak{b}_{q+1}^{n-m} )$ is a partition of intervals such that intervals in $ \mathfrak{g}_{q+1}^{n-m}$ have length $\leq c e^{-(n-m) (h_{q+1} - \epsilon)}$,
\item for $1 \leq i \leq s(q)$, $J_i(q)$ is an interval with $\lambda(J_i(q)) \geq \frac{1}{c} e^{- (n+l) (h_q + \epsilon)}$.
\end{enumerate}
Notice that 
$\frac{1}{c} e^{- (n+l) (h_q + \epsilon)} \geq c e^{-(n-m) (h_{q+1} - \epsilon)}$ due to \eqref{eq:gamma} and $\gamma n \geq 2 \log c$. 

Assume that $\mathfrak{g}_{q+1}^{n-m}$ is comprised of intervals $U_1, \dots, U_s$, where  $U_i$ are subintervals of the interval $[0,1]$ having disjoint interiors. If $\mathfrak{b}_{q+1}^{n-m} = \emptyset$, then $U_1, \dots, U_s$ form a partition $\mathcal{P}$.
Otherwise, there exist intervals $V_1, \dots, V_t$ such that $U_1, \dots, U_s, V_1, \dots V_t$ form a partition $\mathcal{P}$ of $[0,1]$. 

For each interval $J_i(q)$, one can find a finite  family of intervals in $\mathcal{P}$ such that it covers $\overline{J_i(q)}$. Among these, let $K_{i,1}, K_{i,2}$ be two intervals covering respectively left end point and right end point of $\overline{J_i(q)}$, respectively. Thus, there exist a finite family of intervals $I_{i,j}$ $(1 \leq j \leq t_i)$ and two intervals $K_{i,1}$ and $K_{i,2}$ in $\mathcal{P}$ such that 
\[ \bigcup_{j=1}^{t_i} I_{i,j} \subset J_i(q) \subset \left( \bigcup_{j=1}^{t_i} I_{i,j} \right) \bigcup \left( K_{i,1} \bigcup K_{i,2} \right). \]

Let $H_1(q)$ be the union of intervals $I_{i,j}, 1 \leq i \leq s(q), 1 \leq j \leq t_i$ which belong to $\mathfrak{g}_{q+1}^{n-m}$. Let $H_2(q) = I_1(q) \setminus H_1(q)$. 
Then $H_2(q)$ is covered by (some) intervals from $ \mathfrak{b}_{q+1}^{n-m}$ and by at most $2s(q)$  intervals from $\mathfrak{g}_{q+1}^{n-m}$. 
By Corollary \ref{entropy:intervals}, $\lambda (\bigcup_{B \in  \mathfrak{b}_{q+1}^{n-m}} B) \leq c \epsilon$. 
Using \eqref{eq:gamma} and \eqref{estimate:s} we have that
\begin{align*} 
\lambda(H_2(q)) &\leq c \epsilon + 2s(q) c e^{-(n-m) (h_{q+1} - \epsilon)} \\
&\leq  c \epsilon +  2c^2 e^{(n+l) (h_q + \epsilon)} e^{-(n-m) (h_{q+1} - \epsilon)}  
= O(\epsilon + e^{-\gamma n}), 
\end{align*}
so Lemma \ref{sec4:lem2} holds.


\noindent {\em Proof of Lemma \ref{sec4:lem1}.}
We need to show that there are sets $I_1(q+1)$ and $I_2(q+1)$ such that 
\begin{equation}
\label{eqn:lem4.2:goal}
    T_0^{-n} A_0 \cap T_1^{-n} A_1 \cap \cdots \cap T_{q+1}^{-n} A_{q+1} = I_1(q+1) \cup I_2(q+1),
\end{equation}
where 
\begin{enumerate}
\item $I_1(q+1) = \bigcup_{i=1}^{s(q+1)} J_i(q+1)$, where $J_i(q+1) \in \mathfrak{g}_{q+1}^{n+l}$ for $1 \leq i \leq s(q+1)$,
\item $I_2(q+1)$ is a set with $\lambda(I_2(q+1)) = O(\epsilon + e^{-\gamma n})$,
\item $\lambda(I_1(q+1)) \leq c \cdot c_0^{q+1} \prod_{i=0}^{q+1} \mu_i(A_i) + q b_{[\delta n]}$.
\end{enumerate}

Let us first find $I_1(q+1)$ satisfying condition (1).
By Lemma \ref{sec4:lem2}, 
\begin{equation}
\label{eqnum4.9}    
T_0^{-n} A_0 \cap T_1^{-n} A_1 \cap \cdots \cap T_q^{-n} A_q = H_1(q) \cup H_2(q) \cup I_2(q),
\end{equation}
where $H_1(q) \in \sigma ((\mathcal{A}_{q+1})^{n-m})$ and $\lambda (H_2(q) \cup I_2(q)) = O(\epsilon + e^{- \gamma n})$. 
Note that $H_1(q) \bigcap T_{q+1}^{-n} A_{q+1} \in \sigma ((\mathcal{A}_{q+1})^{n+l})$. Thus, by Corollary \ref{entropy:intervals}, we can write 
\begin{equation}
\label{H_1}
H_1(q) \cap T_{q+1}^{-n} A_{q+1} = I_1(q+1) \cup R,
\end{equation}
where $I_1(q+1) = \bigcup\limits_{i=1}^{s(q+1)}J_i(q+1)$ with $J_i(q+1) \in \mathfrak{g}_{q+1}^{n+l}$ and $R$ is the union of some intervals in $\mathfrak{b}_{q+1}^{n+l}$, so $\lambda (R) \leq c \epsilon$. 
Note also that 
$$I_1(q+1) \subset T_0^{-n} A_0 \cap T_1^{-n} A_1 \cap \cdots \cap T_{q+1}^{-n} A_{q+1}.$$
Let 
\begin{equation*}
\label{I2}    
I_2(q+1) = \left(T_0^{-n} A_0 \cap T_1^{-n} A_1 \cap \cdots \cap T_{q+1}^{-n} A_{q+1} \right) \setminus I_1(q+1).
\end{equation*}
Note that 
\[ I_2(q+1) \subset R \cup H_2(q) \cup I_2(q),\]
so $\lambda (I_2(q+1)) = O(\epsilon + e^{-\gamma n})$, which implies that $I_2(q+1)$ satisfies condition (2).

Now let us show that condition (3) holds:
\begin{align*} 
\lambda(I_1(q+1)) &\leq \lambda (H_1(q) \cap T_{q+1}^{-n} A_{q+1}) \quad (\text{by } \eqref{H_1})\\ 
&\leq c_0 \lambda(H_1(q))  \, \mu_{q+1}(A_{q+1}) + b_{[\delta n]} \quad (\text{since } T_{q+1}  \text{ has property B})\\
&\leq c_0 \lambda(I_1(q)) \, \mu_{q+1}(A_{q+1}) + b_{[\delta n]} \quad (\text{by  Lemma } \ref{sec4:lem2}) \\ 
&\leq c \cdot c_0^{q+1} \prod_{i=0}^{q+1} \mu_i(A_i) + q b_{[\delta n]} \quad (\text{by assumption } (iii) \text{ in Lemma \ref{sec4:lem1} })
\end{align*}
This proves Lemma \ref{sec4:lem1} and concludes the proof of Theorem \ref{thm:L^2_joint}.
\end{proof}

The following example is an immediate consequence of Theorem \ref{thm:L^2_joint}.
\begin{ex}[Theorem \ref{Intro-Thm-JE-Entropy2}]
\label{ex:thmintro12}

For any $f_{0}, f_1, f_2, \dots, f_{k+1} \in L^{\infty} (\lambda)$, 
\begin{equation}
\label{eq:ex5.5}
\begin{split} 
\begin{split} 
\lim\limits_{N -M \rightarrow \infty} \frac{1}{N -M } \sum\limits_{n=M}^{N-1} 
&  T_0^n f_{0} \cdot T_{\beta_1}^n f_{1} \cdots T_{\beta_k}^n f_{k}  \cdot T_G^n f_{k+1} \\
 &= \int f_{0} \, d \lambda  \cdot \prod_{i=1}^k \int f_{i} \, d \mu_{\beta_i} \cdot  \int f_{k+1} \, d \mu_G \quad \text{in } L^{2}(\nu),
\end{split}
\end{split}
\end{equation}
where $T_0$ is an ergodic interval exchange transformation and $\beta_1 , \cdots , \beta_t$ are distinct real numbers with $\beta_i >1$ and $\log \beta_i \ne \frac{\pi^2}{6 \log 2}$. 
\end{ex}

The following example demonstrates that the requirement in Theorem  \ref{thm:L^2_joint} that the transformations involved have different entropy (condition (3)) is essential. 
\begin{ex}
\label{ex3.1n}
Define 
\begin{equation*}
 S(x) = \begin{cases} 
      2x  + \frac{1}{2}, \quad & 0 \leq x < \frac{1}{4}, \\
      2x ,                     \quad & \frac{1}{4} \leq x < \frac{1}{2}, \\
      2x -1 ,                     \quad & \frac{1}{2} \leq x < \frac{3}{4}, \\
      2x - \frac{3}{2},  \quad & \frac{3}{4} \leq x < 1.
   \end{cases}
\end{equation*}
Take a partition 
\[ \mathcal{A}_S = \{ [0, 1/4), [1/4, 1/2), [1/2, 3/4), [3/4,1) \}.\]
so that $(S, \lambda, \mathcal{A}_S) \in \mathcal{T}$. Rokhlin formula (Theorem \ref{Rokhlin}) says that $h(S) = \log 2$. 

Let $T_b x = b x \, \bmod \,1 $ with $b \geq 2$ an integer. 
By Theorem \ref{thm:L^2_joint}, if $b \ne 2$, $T_b$ and $S$ are uniformly jointly ergodic. 

However, $T_2$ and $S$ are not jointly ergodic. 
Indeed, for $f(x) = e^{4 \pi i x }$ and $g(x) = e^{- 4 \pi i x}$, we have that 
$f (T_2^n x) \cdot g(S^n x) = 1 \quad \text{for all } n$, so 
$ \int f (T_2^n x) \cdot g(S^n x) \, d \lambda = 1$,
whereas $ \int f \, d \lambda \, \int g \, d  \lambda = 0$.
\end{ex}

Example \ref{ex3.1n} notwithstanding, it well can be that answer to the following question is positive. (See also Question \ref{Que:SkewTent} below.)
\begin{Question}
\label{Que:sameent}
If $\log \beta = \frac{\pi^2}{6 \log 2}$, then $h(T_{\beta}) = h(T_G)$. Are $T_{\beta}$ and $T_G$ uniformly jointly ergodic? jointly mixing?
\end{Question} 

\section{Joint mixing}
\label{sec:jointmixingPM}

In this section, we establish useful criteria for joint mixing and joint weak mixing for transformations in $\mathcal{T}$ and obtain some interesting applications to skew-tent maps and restrictions of Blaschke products to the unit circle. 

The following theorem has similar formulation to Theorem \ref{thm:L^2_joint}, the main distinction being that the property B which appears in the formulation of Theorem \ref{thm:L^2_joint} is now replaced by $\alpha$-mixing (a stronger condition which leads to stronger conclusions). 

\begin{Theorem}[Theorem \ref{Int:Thm:JM2:new}]
\label{Thm:JM2:new}
For $i = 0, 1, 2, \dots, k$, let $\mu_i$ be a probability measure on the measurable space $([0,1], \mathcal{B})$ and let  $T_i:[0,1] \rightarrow [0,1]$ be $\mu_i$-preserving ergodic transformations. 
Suppose that
\begin{enumerate}
\item for $i = 0, 1, 2, \dots, k$, there is a partition $\mathcal{A}_i$ with  $(T_i, \mathcal{A}_i, \mu_i) \in \mathcal{T}$,
\item for $i = 1, 2, \dots, k$, $(T_i, \mathcal{A}_i, \mu_i)$ is $\alpha$-mixing,
\item $0 \leq h(T_0) < h(T_1) < \cdots < h(T_k) < \infty$, where $h(T_i)$ denotes the entropy of $T_i$.
\end{enumerate} 
Let $\nu$ be any probability measure on $([0,1], \mathcal{B})$, which is equivalent to $\lambda$.
Then 
\begin{enumerate}[(i)]
\item if $T_0$ is mixing, then for any $B, A_0, \dots, A_k \in \mathcal{B}$,
\begin{equation}
\label{Jm:eqn:24}
 \lim_{n \rightarrow \infty} \nu ( B \cap T_0^{-n} A_0 \cap T_1^{-n}A_1 \cap \cdots \cap T_k^{-n}A_k) =   \nu (A_0) \prod_{i=0}^k \mu_i(A_i) \tag*{(\theequation)$_{k}$}\refstepcounter{equation}
\end{equation}
\item if $T_0$ is weakly mixing, then for any $B, A_0, \dots, A_k \in \mathcal{B}$,
\begin{equation*}
 \lim_{N-M \rightarrow \infty} \frac{1}{N-M} \sum_{n=M}^{N-1} | \nu (B \cap T_0^{-n} A_0 \cap T_1^{-n}A_1 \cap \cdots \cap T_k^{-n}A_k) - \nu (B) \prod_{i=0}^k \mu_i(A_i)| = 0. 
\end{equation*}
\end{enumerate}
\end{Theorem}

Before embarking on the proof, we formulate and prove a useful auxiliary lemma. 
\begin{Lemma}
\label{Lem:Conv:Weak}
Suppose that  $\nu$ is a probability measures on  $(X, \mathcal{B})$, which is equivalent to $\mu$.
If $T_1, \dots, T_k$ are jointly mixing with respect to $(\mu; \mu_1, \dots, \mu_k)$, then $T_1, \dots, T_k$ are jointly mixing with respect to $(\nu; \mu_1, \dots, \mu_k)$.
\end{Lemma}
\begin{proof}
We need to show that for any measurable sets $B_0, \dots, B_k$, 
\begin{equation}
\label{eqn:weakconv} 
\lim_{n \rightarrow \infty} \nu (B_0 \cap T_1^{-n} B_1 \cap \cdots \cap T_k^{-n} B_k) = \nu (B_0) \prod_{i=1}^k \mu_i (B_i).
\end{equation}
Let $f_i = 1_{B_i}$ for $i = 0, \dots, k$. 
Define $A_{\delta} = \{ x: \frac{d \nu}{d \mu}(x) > \frac{1}{\delta} \}$. 
Then we have that
\begin{align*}
\nu &(B_0 \cap T_1^{-n} B_1 \cap \cdots \cap T_k^{-n} B_k)) \\
& = \int f_0(x) \cdot T_1^n f_1(x) \cdots T_k^n f_k(x) \, \frac{d \nu}{d \mu}(x) \, d \mu \\
&= \int 1_{A_{\delta}}(x)  \frac{d \nu}{d \mu}(x) f_0(x) \prod_{i=1}^k f_i(x) \, d \mu + \int 1_{A_{\delta}^c}(x)  \frac{d \nu}{d \mu}(x) f_0(x) \prod_{i=1}^k f_i(x) \, d \mu. 
\end{align*}
Note that 
\[ 0 \leq  \int 1_{A_{\delta}}(x)  \frac{d \nu}{d \mu}(x) f_0(x) \prod_{i=1}^k f_i(x) \, d \mu   \leq \nu(A_{\delta}) \]
and 
\begin{align*} 
\lim_{n \rightarrow \infty} \int 1_{A_{\delta}^c}(x)  \frac{d \nu}{d \mu}(x) f_0(x) \prod_{i=1}^k f_i(x) \, d \mu 
&= \int 1_{A_{\delta}^c}(x)  \frac{d \nu}{d \mu}(x) f_0(x) d \mu \prod_{i=1}^k \int f_i \, d \mu_i \\
&= \int_{A_{\delta}^c} f_0(x) d \nu \prod_{i=1}^k \int f_i \, d \mu_i. 
\end{align*}
Since $\nu(A_{\delta}) \rightarrow 0$ and  $\int_{A_{\delta}^c} f_0(x) d \nu \rightarrow \int f_0(x) d \nu = \nu(B_0)$, as $\delta \rightarrow 0$,
\eqref{eqn:weakconv} follows.
\end{proof}

\begin{proof}[Proof of Theorem \ref{Thm:JM2:new}]
We will prove $(i)$. The proof of item $(ii)$ goes along similar lines and is omitted.  

Since $T_0$ is mixing, we have
\[ \lim_{n \rightarrow \infty} \mu_0 (B \cap T_0^{-n}A_0) = \mu_0 (B) \mu_0(A_0).\]
It follows then from Lemma \ref{Lem:Conv:Weak} that 
\[ \lim_{n \rightarrow \infty} \nu (B \cap T_0^{-n}A_0) = \nu (B) \mu_0(A_0),\]
which gives formula \ref{Jm:eqn:24} for $k=0$.

We will prove that for $0 \leq q \leq k-1$, the validity \ref{Jm:eqn:24} for $k=q$ implies the validity of \ref{Jm:eqn:24} for $k=q+1$. 
In view of Lemma \ref{Lem:Conv:Weak}, we will prove this for $\nu = \mu_{q+1}$:
 for any measurable sets $B, A_0, \dots, A_{q+1} \in \mathcal{B}$,
\begin{equation}
\label{sec5:thm:pf:eqn1}
 \lim_{n \rightarrow \infty} \mu_{q+1} ( B \cap T_0^{-n} A_0 \cap T_1^{-n}A_1 \cap \cdots \cap T_{q+1}^{-n}A_{q+1}) =   \mu_{q+1} (B) \prod_{i=0}^{q+1} \mu_i(A_i).
\end{equation}
Observe that it is enough to show that for any $l \in \mathbb{N}$, \eqref{sec5:thm:pf:eqn1} holds for any measurable sets $B, A_0, \dots, A_{q+1}$ with $B \in (\mathcal{A}_0)^l$ and  $A_i \in (\mathcal{A}_i)^l$, $i = 0, 1, 2, \dots, q+1$.

As in the proof of Theorem \ref{thm:L^2_joint}, write $h_i = h(T_i)$ and choose  $\epsilon > 0$ so that $h_i - h_{i-1} - 2 \epsilon > 0$ for $i= 1, 2, \dots, q+1$. Also, choose $\delta, \gamma >0$ so that for all $i = 1, 2, \dots, q+1$,
if $n$ is sufficiently large, one has 
\begin{equation}
\label{sec5.1:eq:gamma}
(h_{i} - \epsilon) (1- \delta)n  - (h_{i-1} + \epsilon)(n+l)  > \gamma n.
\end{equation}
As before, we will be using the following notation: for $i = 0, 1 , \dots, q+1$,
\begin{itemize}
\item $\mathfrak{g}_i^n = \{A \in (\mathcal{A}_i)^n:   e^{-n (h_i + \epsilon)} < \mu_i(A) <  e^{- n (h_i - \epsilon)}\}$,
\item $\mathfrak{b}_i^n = (\mathcal{A}_i)^n  \setminus \mathfrak{g}_i^n$. 
\end{itemize}
Corollary \ref{entropy:intervals} says that there is $N_0 = N_0 (\epsilon)$ such that if $n \geq N_0$, then $\mu_i (\bigcup_{B \in  \mathfrak{b}_i^n} B) \leq \epsilon$. For the rest of proof, we will be assuming that $n$ is sufficiently large so that \eqref{sec5.1:eq:gamma} holds, $n - [\delta n] \geq N_0$ and $\gamma n \geq 2 \log c$.

We will need the following lemma, which plays the same role as Lemmas \ref{sec4:lem1} and \ref{sec4:lem2} play in the proof of Theorem \ref{thm:L^2_joint}. The proof is similar and is omitted.
\begin{Lemma}
\label{lem:sec5:new}
Let $l \in \mathbb{N}$ be fixed. Let $B, A_0, \dots, A_{q}$ with $B \in (\mathcal{A}_0)^l$ and  $A_i \in (\mathcal{A}_i)^l$, $i = 0, 1, 2, \dots, q$.
If $n$ is sufficiently large, then there are disjoint sets $H(q)$ and $J(q)$  such that 
\begin{equation}
\label{sec5.1:eqn5.5}
    B \cap T_0^{-n} A_0 \cap T_1^{-n} A_1 \cap \cdots \cap T_{q}^{-n} A_{q} = H(q) \cup J(q),
\end{equation}
where 
\begin{itemize}
\item $H(q)$ is union of some intervals in $\mathfrak{g}_{q+1}^{n-[\delta n]}$, 
\item $\lambda (J(q)) = O (\epsilon + e^{-\gamma n})$.
\end{itemize}
\end{Lemma}

Since $(T_{q+1}, \mu_{q+1}, \mathcal{A}_{q+1})$ is $\alpha$-mixing,
there exists a sequence $\alpha(n)$ with $\alpha(n) \rightarrow 0$ such that  for any $p \geq 0$, if $C \in  \sigma ((\mathcal{A}_{q+1})^p)$ and $D  \in \mathcal{B}$,
\[ | \mu_{q+1} (C \cap T_{q+1}^{-(n+p)} D) - \mu_{q+1} (C) \mu_{q+1} (D) | \leq \alpha(n) \,\, \text{ for any } n \geq 1.\]
Then we have that
\begin{align*} 
&\mu_{q+1} ( B \cap T_0^{-n} A_0 \cap T_1^{-n}A_1 \cap \cdots \cap T_{q+1}^{-n}A_{q+1}) \\ 
&= \mu_{q+1} (H(q) \cap T_{q+1}^{-n} A_{q+1}) + O(\epsilon + e^{-\gamma n}) \quad (\text{by } \eqref{sec5.1:eqn5.5} )\\ 
&= \mu_{q+1}(H(q))  \, \mu_{q+1}(A_{q+1}) + \alpha ([\delta n]) + O(\epsilon + e^{- \gamma n}) \quad (\text{since } T_{q+1} \text{ is  $\alpha$-mixing})\\
&= \mu_{q+1} (B \cap T_0^{-n} A_0 \cap T_1^{-n} A_1 \cap \cdots \cap T_{q}^{-n} A_{q})   \, \mu_{q+1}(A_{q+1})  + \alpha ([\delta n])  +   O(\epsilon + e^{-\gamma n}) \quad (\text{by } \eqref{sec5.1:eqn5.5}) 
\end{align*}

By the assumption, we have that for any measure $\nu$, which is equivalent to $\lambda$,
$$\lim_{n \rightarrow \infty} \nu (B \cap T_0^{-n} A_0 \cap T_1^{-n} A_1 \cap \cdots \cap T_{q}^{-n} A_{q})  = \nu(B) \prod_{i=0}^q \mu_i (A_i),$$ 
so 
$$\lim_{n \rightarrow \infty} \mu_{q+1} (B \cap T_0^{-n} A_0 \cap T_1^{-n} A_1 \cap \cdots \cap T_{q}^{-n} A_{q})  = \mu_{q+1}(B) \prod_{i=0}^q \mu_i (A_i).$$ 
Therefore, 
\[
 \limsup_{n \rightarrow \infty} | \mu_{q+1} (B \cap T_0^{-n} A_0 \cap T_1^{-n} A_1 \cap \cdots \cap T_{q+1}^{-n} A_{q+1}) - \mu_{q+1} (B) \prod_{j=0}^{q+1} \mu_j (A_j) | =   O ( \epsilon),
\]  
Since $\epsilon$ is arbitrarily small, 
\[
 \lim_{n \rightarrow \infty}   \mu_{q+1}  (B \cap T_0^{-n} A_0 \cap T_1^{-n} A_1 \cap \cdots \cap T_{q+1}^{-n} A_{q+1}) =  \mu_{q+1} (B) \prod_{j=0}^{q+1} \mu_j (A_j). 
\]
The validity of  \eqref{sec5:thm:pf:eqn1} for general $B, A_0, \dots, A_{q+1} \in \mathcal{B}$  follows now by the standard approximation argument.
\end{proof}

The following example is an immediate consequence of Theorem \ref{Thm:JM2:new}:
\begin{ex}
Let $T_1, \dots, T_k$ be transformations belonging to $\mathcal{T}$ with positive distinct entropies. For $i = 1, 2, \dots, k$, let $\mu_i$ be a $T_i$-preserving ergodic measure, which is equivalent to Lebesgue measure $\lambda$. 
\begin{enumerate}
\item If $S$ is a mixing rank one transformation belonging to $\mathcal{T}$, then 
for any measurable sets $ B, A_0, A_1, \dots, A_k$, 
\begin{equation*}
\lim\limits_{n \rightarrow \infty} \lambda (B \cap S^{-n} A_0 \cap T_{1}^{-n} A_{1} \cap \cdots \cap T_{k}^{-n} A_{k}) =  \lambda (B) \lambda(A_0) \prod_{i=1}^k \mu_i (A_{i}).
\end{equation*}
\item If $S$ is a weak mixing interval exchange transformation or a weak mixing rank one transformation belonging to $\mathcal{T}$ (for instance, Chacon's transformation), then for any measurable sets $B, A_0, A_1, \dots, A_k$, 
\begin{equation*}
\lim\limits_{N-M \rightarrow \infty} \frac{1}{N-M} \sum_{n=M}^{N-1} |\lambda (B \cap S^{-n} A_0 \cap  T_{1}^{-n} A_{1} \cap \cdots \cap T_{k}^{-n} A_{k}) -  \lambda (B) \lambda(A_0) \prod_{i=1}^k \mu_i (A_{i})|=0.
\end{equation*}
\end{enumerate}
\end{ex}

The following example involves the so-called skew tent maps.  
\begin{ex}
\label{ex:skewtent}
For $0 < a <1$, let
\[T_{a} (x) 
= \begin{cases} 
      \frac{x}{a}, &  0 \leq x \leq a,  \\
    \frac{1-x}{1- a}, & a <  x \leq 1.
   \end{cases}
\]

Note that each $T_a$ is two-to-one, preserves the Lebesgue measure, and satisfies the conditions of Theorem  \ref{folklore}. By Rokhlin formula, we have that $h(T_{a}) = - a \log a - (1- a) \log (1-a)$. Thus if $a_1, \cdots, a_k \in (0,1)$ with $a_i \ne a_j$ and $a_i \ne 1- a_j$ for $i \ne j$, then $T_{a_1}, \dots, T_{a_k}$ are jointly mixing. 
We do not know the answer to the following question.
\begin{Question}
\label{Que:SkewTent}
 Are $T_{a}$ and $T_{1-a}$ jointly mixing? (Note that $h(T_a) = h(T_{1-a})$.)
\end{Question}
\end{ex}

\begin{ex}
This example deals with restrictions of finite Blaschke products to the unit circle $S^1$. (See Subsection \ref{blaschke}.)
For $l = 1, 2, \dots, k$, let
\[ f_l (z) = C_l \prod_{j=1}^{M_l} \frac{z - a_j^l}{1- \overline{a_j^l} z},\]
where $|C_l| = 1$, $|a_j^l| < 1$ for $j = 1, 2, \dots, M_l$ and $M_l >1$.
Suppose that for $l = 1, 2, \dots, k$,
\[ \sum_{j=1}^{M_l} \frac{1 - |a_j^l|}{1+ |a_j^l|} > 1.\]

For $l = 1, 2, \dots, k$, let $z_l$ be the unique fixed point of $f_l$ in the unit disk and let $\tau_l$ be the restriction of Blaschke product $f_l$ on the circle $S^1$. 

If entropy values $h(\tau_1), \dots, h(\tau_k)$ are distinct, then $\tau_1, \dots, \tau_k$ are jointly mixing: for any Borel measurable sets $A_0, A_1, \dots, A_k \subset S^1$, 
\begin{equation}
\label{eq:sm:Bla}
    \lim\limits_{n \rightarrow \infty}  \sigma( A_0 \cap \tau_1^{-n} A_1 \cap \cdots \cap \tau_k^{-n} A_k) = \sigma (A_0) \prod_{l=1}^k \int_{S^1} 1_{A_l} (u) \frac{1- |z_l|^2}{|u - z_l|^2} d \sigma (u),
\end{equation}
where $\sigma$ is the normalized Lebesgue measure on $S^1$. 

Note that in the light of Remark \ref{rem:mixje}, formula \eqref{eq:sm:Bla} implies that 
for any $g_1, \dots g_k \in L^{\infty} (S^1)$,
\[ \lim\limits_{N -M \rightarrow \infty} \frac{1}{N -M } \sum\limits_{n=M}^{N-1} \prod_{l=1}^k g_l \circ \tau_l^n = \prod_{l=1}^k \int_{S^1} g_l(u) \frac{1- |z_l|^2}{|u - z_l|^2} d \sigma (u) \quad \text{in } L^2 (\sigma).\]
\end{ex}

\section{Uniform versus Non-uniform joint ergodicity}
\label{unifvsnon}
In this section we consider non-uniform variant of joint ergodicity.
\begin{Definition}
Measure preserving transformations $T_1, \dots, T_k$ on a probability space $(X, \mathcal{B}, \mu)$ are {\em jointly ergodic} if for any $f_1, \dots, f_k \in L^{\infty}$, 
\begin{equation*}
    \lim\limits_{N \rightarrow \infty} \frac{1}{N } \sum\limits_{n=0}^{N-1}  T_1^{n} f_1 \cdot T_2^n f_2 \cdots T_k^n f_k = \prod_{i=1}^k \int f_i \, d \mu \quad \text{ in } L^2(\mu).
\end{equation*}
\end{Definition}

One can show that if $T_1, \dots, T_k$ are commuting and invertible, then the notions of joint ergodicity and uniform joint ergodicity coincide (\cite{BeBe1}). However, in general, these are different notions, which is demonstrated by the following two examples.

\begin{ex}[Example 3.1 in \cite{BeBe2}]
Let $G = C_2^{\mathbb{Z}}$, where $C_2$ is a cyclic group of order 2. Let $\mu$ be the probability Haar measure on $G$.
Let $T$ be the left shift on $G$:
\[ (Tx)_k = x_{k+1}, \,\, x = (x_k)_{k \in \mathbb{Z}} \in G, k \in \mathbb{Z}.\]

Let $A \subset \mathbb{Z}$ be a symmetric set with 
\[ d(A) = \lim_{N \rightarrow \infty} \frac{1}{2N+1} |A \cap \{-N, -N+1, \dots, N\}| = 0,\]
and 
\[d^*(A) = \limsup_{N -M \rightarrow \infty} \frac{1}{N - M + 1}  |A \cap \{M, M+1, \dots, N\}| = 1.\]
(For instance, take  $A = \bigcup_{n=1}^{\infty} \left( [n^3, n^3 + n] \cup [-n^3-n, -n^3] \right).$)

Define $S: G \rightarrow G$ by $(Sx)_k = x_{\pi(\pi(k) + 1 )}$ for $x = (x_k)_{k \in \mathbb{Z}} \in G $, where
\[ \pi (k) = \begin{cases} 
       k, & k \in  A, \\
     -k, &  k \notin A.
   \end{cases}
\]
Then $T, S$ are jointly ergodic, but not uniformly jointly ergodic. See \cite{BeBe2} for the details.
\end{ex}

Before presenting next example, recall that a sequence $(x_n)$ in $\mathbb{R}$ is said to be 
\begin{enumerate}
\item {\em uniformly distributed $\bmod \, 1$} if for any $a, b$ with $0 \leq a < b <1$,
$$ \frac{1}{N} \big| \{ 1 \leq n \leq N: (x_n  \bmod 1) \in [a,b) \} \big| \rightarrow b-a,$$ 
\item {\em well-distributed $\bmod \, 1$} if for any $a, b$ with $0 \leq a < b <1$,
$$ \frac{1}{N} \big| \{ k+1 \leq n \leq N+k: (x_n \bmod 1) \in [a,b) \} \big| \rightarrow b-a \quad \text{uniformly in } k = 0, 1, 2, \dots.$$ 
\end{enumerate}
These notions can be naturally extended to sequences in $\mathbb{R}^d$. 
We use the following facts to provide another example of jointly ergodic transformations, which are not uniformly jointly ergodic.
\begin{itemize}
\item For a.e. $k$-tuple $(\alpha_1, \dots, \alpha_k)$, $(2^n \alpha_1, \dots, 2^n \alpha_k)_{n \in \mathbb{N}}$ is uniformly distributed $\bmod \, 1$. For $k=1$, this is a classical result due to Borel (\cite{Bor}) (see also \cite{KN}, Chapter 1, Corollary 8.1). For $k \geq 2$, this follows from Weyl's criterion. One can also obtain this result by applying the pointwise ergodic theorem to the $k$-fold product transformation $T \times \cdots \times T$, where $Tx = 2x \, \bmod 1$. 
\item For any $\alpha$, $2^n \alpha$ is not well-distributed $\bmod \, 1$.
(See \cite{KN}, Chapter 1, Theorem 5.3.)
\end{itemize}

\begin{ex} 
\label{ex4}
Let $\alpha_0 = 0$ and let $(\alpha_1, \dots, \alpha_k) \in \mathbb{R}^k$ be such that $(2^n \alpha_1, \dots, 2^n \alpha_k)_{n \in \mathbb{N}}$ is uniformly distributed $\bmod \, 1$.
For each $j = 0, 1, \dots, k$, define $T_j: [0,1] \rightarrow [0,1]$  by 
$$T_j(x) = (2x + \alpha_j) \, \bmod 1, \quad x \in [0,1].$$ 
Then,  for $j = 0, 1, \dots, k$, 
$$T_j^n(x) = (2^n x + (2^n - 1) \alpha_j) \, \bmod \, 1, \quad  x\in [0,1].$$
Notice that for $f_j (x) = \exp(2 \pi i h_j x)$ with $h_j \in \mathbb{Z}$ for $0 \leq j \leq k$,
\begin{align*} 
\int \prod_{j=0}^k f_j (T_j^n x) dx &= \exp(2 \pi i \sum_{j=1}^k h_j (2^n -1)\alpha_j) \int \exp (2 \pi i \sum_{j=0}^k h_j 2^n x) \, dx \\
& = \begin{cases} 
      \exp(2 \pi i \sum_{j=1}^k h_j (2^n -1)\alpha_j), & h_0 + h_1 + \cdots +h_k = 0, \\
      0, &  h_0 + h_1 + \cdots +h_k \ne 0.
   \end{cases}
\end{align*}
If $(h_0, h_1,  \cdots,  h_k) \ne (0, 0, \cdots , 0)$, then $h_0 + h_1 + \dots + h_k \ne 0$ or some $h_i \ne 0$ for $1 \leq i \leq k$, so we have that
\begin{equation*} 
\lim_{N \rightarrow \infty} \frac{1}{N} \sum_{n=0}^{N-1} \int \prod_{j=0}^k f_j (T_j^n x) dx  = 
\begin{cases} 
      0, & (h_0, h_1,  \cdots,  h_k) \ne (0, 0, \cdots , 0), \\
      1, &  h_0 = h_1 = \cdots = h_k =0.
   \end{cases}
\end{equation*}
Since the set of trigonometric polynomials is dense in $L^{\infty}(\lambda)$,  for any $g_0, g_1, \dots, g_k \in L^{\infty}$,
\begin{equation}
    \label{je.eq}
    \lim_{N \rightarrow \infty} \frac{1}{N} \sum_{n=0}^{N-1} \int \prod_{j=0}^k g_j (T_j^n x)   dx =\prod_{j=0}^k \int g_j(x) \, dx.
\end{equation}
Since $T_0 \times T_1 \times \cdots \times T_k$ are ergodic, it follows from Theorem 2.1 in \cite{BeBe2} that  formula \eqref{je.eq}  implies that 
$T_0, T_1, \dots, T_k$ are jointly ergodic. (This fact also follows from Theorem \ref{criterion:JE2} below). 

However, $T_0, T_1$ are not uniformly jointly ergodic and so $T_0, T_1, \dots, T_k$ are not uniformly jointly ergodic, hence, in view of Remark \ref{rem:mixje}, they are not jointly mixing.
Indeed, for $f_0(x) = \exp (- 2 \pi i h x)$ and $f_1(x) = \exp(2 \pi i h x)$, where $h$ is a non-zero integer, the limit of the following average 
\begin{equation*}
    \label{eq:nonuniform;ex}
    \frac{1}{N - M} \sum_{n=M}^{N-1} \int f_0 (T_0^{n} x) f_1(T_1^{n} x) \, dx 
= \frac{1}{N- M} \sum_{n=M}^{N-1} \exp (2 \pi i h (2^n -1) \alpha_1).
\end{equation*} 
does not converge to $0$ for some $h$ as $N -M \rightarrow \infty$, since $2^n \alpha_1$ is not well-distributed.
\end{ex}

\begin{Remark}
Example \ref{ex4} shows that Theorems \ref{thm:L^2_joint} and \ref{Thm:JM2:new} may fail if the involved transformations do not have distinct entropies.
Indeed, $T_0, T_1$ are  mixing transformations belonging to $\mathcal{T}$ with $h(T_0) = h(T_1) = \log 2$, but they are not uniformly jointly ergodic, and not jointly mixing. 
\end{Remark}

Now we extend the notion of joint ergodicity to the case when the involved measure preserving transformations have potentially different invariant measures. 
\begin{Definition}
Let $(X, \mathcal{B}, \mu)$ be a probability space. Let $T_1, \dots, T_k$ be measurable transformations on $(X, \mathcal{B})$ such that for each $i$, there is a $T_i$-invariant probability measure $\mu_i$, which is equivalent to $\mu$.
We say that $T_1, \dots, T_k$ are {\em jointly ergodic} with respect to $(\mu; \mu_1, \dots, \mu_k)$ if for any $f_1, \dots, f_k \in L^{\infty}$, 
$$\lim\limits_{N \rightarrow \infty} \frac{1}{N } \sum\limits_{n=0}^{N-1}  T_1^{n} f_1 \cdot T_2^n f_2 \cdots T_k^n f_k = \prod_{i=1}^k \int f_i \, d \mu_i \quad \text{ in } L^2(\mu).$$
\end{Definition}

In view of the fact that joint ergodicity and uniform joint ergodicity are different notions, it is desirable to have criteria for joint ergodicity similar to Theorem \ref{criterion:JE}.
An analogue of Theorem \ref{criterion:JE} is furnished by the following theorem.
The proof is analogous to the proof of Theorem \ref{criterion:JE} and is omitted.
\begin{Theorem}
\label{criterion:JE2}
Let $(X, \mathcal{B}, \mu)$ be a probability space. Let $T_1, \dots, T_k$ be measurable transformations on $(X, \mathcal{B})$ such that for each $i$, there is a $T_i$-invariant probability measure $\mu_i$, which is equivalent to $\mu$.
The following conditions are equivalent:
\begin{enumerate}[(i)]
\item $T_1, \dots, T_k$ are joint ergodic with respect to $(\mu; \mu_1, \dots, \mu_k)$. 
\item for any bounded measurable functions $f_0, f_1, \dots f_k$, 
\begin{equation*}
\lim_{N \rightarrow \infty} \frac{1}{N} \sum_{n=1}^{N} \int f_0 T_1^{n} f_1 \cdot T_2^n f_2 \cdots T_k^n f_k \, d \mu = \int f_0 \, d \mu \cdot \prod_{i=1}^k \int f_i \, d \mu_i. 
\end{equation*} 
\item \begin{enumerate}
\item $T_1 \times T_2 \times \cdots \times T_k$ is ergodic with respect to $\mu_1 \times \mu_2 \times \cdots \times \mu_k$
\item for  any bounded measurable functions $f_1, \dots f_k$,
\begin{equation*}
\lim_{N \rightarrow \infty} \frac{1}{N} \sum_{n=1}^{N}  \int T_1^{n} f_1 \cdot T_2^n f_2 \cdots T_k^n f_k \, d \mu = \prod_{i=1}^k \int f_i \, d \mu_i .
\end{equation*}
\end{enumerate}
\item \begin{enumerate}
\item $T_1 \times T_2 \times \cdots \times T_k$ is ergodic with respect to $\mu_1 \times \mu_2 \times \cdots \times \mu_k$
\item there is $C>0$ such that for any $A_1, A_2, \dots, A_k \in \mathcal{B}$
\begin{equation*}
\limsup_{N \rightarrow \infty} \frac{1}{N} \sum_{n=1}^{N} \mu (T_1^{-n} A_1 \cap T_2^{-n} A_2 \cap \cdots \cap T_k^{-n} A_k) \leq C \prod_{i=1}^k \mu_i(A_i) .
\end{equation*}
\end{enumerate}
\end{enumerate} 
\end{Theorem}

The following example illustrates Theorem \ref{criterion:JE2}. 
\begin{ex}
Let $T_1(x) = 2x \bmod 1$ and $T_2(x) = (2x + \alpha) \bmod 1$, where $(2^n \alpha)_{n \in \mathbb{N}}$ is uniformly distributed $\bmod \, 1$. Let $T_{\beta} x = \beta x \bmod 1$, where $\beta > 1$, $\beta \ne 2$ and $\log \beta \ne \frac{\pi^2}{ 6 \log 2}$. 
Then $T_1, T_2, T_G, T_{\beta}$ are jointly ergodic (but not uniformly jointly ergodic) with respect to $(\lambda; \lambda, \lambda, \mu_G, \mu_{\beta})$.

Since $T_1 \times T_2 \times T_{G} \times T_{\beta}$ is ergodic with respect to $\lambda \times \lambda \times \mu_G \times \mu_{\beta}$, by Theorem \ref{criterion:JE2} (iii),
the result in question will follow if we show that for any $f_1, f_2, f_3, f_4 \in L^{\infty} (\lambda)$, 
\begin{equation}
\label{eqn:7.3}
\lim_{N \rightarrow \infty} \frac{1}{N} \sum_{n=1}^N \int f_1 (T_1^n x) f_2 (T_2^n x) f_3 (T_G^n x) f_4 (T_{\beta}^n x) \, d \lambda 
= \int f_1 \, d \lambda \int f_2 \, d \lambda \int f_3 \, d \mu_G \int f_4 \, d \mu_{\beta}.
\end{equation}

By the standard approximation argument, it is enough to show that \eqref{eqn:7.3} holds for $f_1 = \exp (2 \pi i h_1 x)$ and $f_2 = \exp (2 \pi h_2 x)$ for any $h_1, h_2 \in \mathbb{Z}$.  Note that
\begin{equation*}
\begin{split}
&\int f_1 (T_1^n x) f_2 (T_2^n x) f_3 (T_G^n x) f_4 (T_{\beta}^n x) \,  d \lambda  \\
&= \exp ( 2 \pi i h_2 (2^n -1) \alpha ) \int \exp (2 \pi i (h_1 + h_2) 2^n x) f_3 (T_G^n x) f_4 (T_{\beta}^n x) \, d \lambda \\
&= \exp ( 2 \pi i h_2 (2^n -1) \alpha ) \int g(T_2^n x)  f_3 (T_G^n x) f_4 (T_{\beta}^n x) \, d \lambda,
\end{split}
\end{equation*} 
where $g(x) = \exp (2 \pi i (h_1 + h_2) x)$.
Since $T_2, T_G, T_{\beta}$ are jointly mixing, 
\begin{equation*}
\lim_{n \rightarrow \infty} \int g(T_2^n x)  f_3 (T_G^n x) f_4 (T_{\beta}^n x) \, d \lambda
= \int g  \, d \lambda \int f_3 \, d \mu_G \int f_4 \, d \mu_{\beta}.
\end{equation*}
Note that if  $h_1 + h_2 \ne 0$, then $\int g \, d \lambda =0$ and if $h_2 \ne 0$, 
\[ \lim_{N \rightarrow \infty} \frac{1}{N} \sum_{n=1}^N \exp ( 2 \pi i h_2 (2^n -1) \alpha ) = 0,\]
since $(2^n \alpha)_{n \in \mathbb{N}}$ is uniformly distributed $\bmod \, 1$.
Therefore if $h_1 + h_2 \ne 0$ or $h_2 \ne 0$, (equivalently if $(h_1, h_2) \ne (0,0)$), then
\[ \lim_{N \rightarrow \infty} \frac{1}{N} \sum_{n=1}^N \int f_1 (T_1^n x) f_2 (T_2^n x) f_3 (T_G^n x) f_4 (T_{\beta}^n x)  d \lambda = 0, \]
which implies \eqref{eqn:7.3}.  This shows that $T_1, T_2, T_G, T_{\beta}$ are jointly ergodic with respect $(\lambda; \lambda, \lambda, \mu_G, \mu_{\beta})$. 
However, as in Example \ref{ex4}, $T_1$ and $T_2$ are not uniformly jointly ergodic, and so $T_1, T_2, T_G, T_{\beta}$ are not uniformly jointly ergodic.
\end{ex}

\section{Disjointness and joint ergodicity}
\label{sec:disjoint}

In this section, we establish additional results on joint ergodicity which utilize the notion of disjointness of dynamical systems that was introduced by Furstenberg in \cite{Fur}. We remark that while in the previous sections we were mostly dealing with the maps of the interval, the main results in this section is quite general. 

\begin{Definition}
Measure preserving systems $(X_i, \mathcal{B}_i, \mu_i, T_i), 1 \leq i \leq k,$ are {\em mutually disjoint} if the only $T_1 \times \cdots \times T_k$-invariant measure on $X_1 \times \cdots \times X_k$, which has the property that the projection of $\mu$ onto $X_i$ is $\mu_i$ for all $i = 1, \dots, k$, is the product measure $\mu_1 \times \cdots \times \mu_k$.
\end{Definition}

By slight abuse of language, we will often say that $T_1, \dots, T_k$
are mutually disjoint if it is clear from the context what are the systems $(X_i,\mathcal{B}_i, \mu_i)$. 

The following simple fact, which will be needed in the sequel, is well known to aficionados. We include the proof for the convenience of the reader.
\begin{Lemma}
\label{lem:dis:prod:erg}  
If measure preserving systems $\mathsf{X}_i = (X_i, \mathcal{B}_i, \mu_i, T_i), 1 \leq i \leq k$, are ergodic and mutually disjoint, then the product system $\mathsf{X}_1 \times \cdots \times \mathsf{X}_k$ is ergodic.  
\end{Lemma}
\begin{proof}
Let us first show that $\mathsf{X}_1 \times \mathsf{X}_2$ is ergodic. 
Since $\mathsf{X}_1$ and $\mathsf{X}_2$ are disjoint, they do not have a nontrivial common factor (see Proposition I.2 in \cite{Fur}). 
So, in particular, $\text{spec } (T_1) \, \bigcap \text{ spec } (T_2) = \{1\}$, where  $\text{spec}(U)$ denotes the set of eigenvalues of $U$.
This, in turn, implies that  $\mathsf{X}_1 \times \mathsf{X}_2$ is an ergodic system. (c.f. Lemma \ref{prop2.1}.) 

Applying the same argument to systems $\mathsf{Y}_1 = \mathsf{X}_1 \times \mathsf{X}_2$ and $\mathsf{Y}_2 = \mathsf{X}_{3}$, we get that $\mathsf{X}_1 \times \mathsf{X}_2 \times \mathsf{X}_3$ is ergodic. And so on.  
\end{proof}

It is not hard to see that if $T_i $ are irrational rotations on $\mathbb{T}$, $T_i(x) = x + \alpha_i$, $i = 1, 2, \dots, k$, where $1, \alpha_1, \dots, \alpha_k$ are rationally independent, then $T_1, \dots, T_k$ are mutually disjoint. This is a special case of Proposition \ref{prop:dis} below. 

We recall first some pertinent definitions.
\begin{Definition}
Let $(X,d)$ be a compact metric space and let $T:X \rightarrow X$ be a homeomorphism. 
A topological dynamical system $(X,T)$ is {\em distal} if 
\[ \inf \{ d(T^n x, T^n y) : n \in \mathbb{Z} \} > 0 \quad \text{for any } x \ne y.\]
\end{Definition}

A measure-theoretic analogue of distality was introduced by Parry in \cite{Pa3}.
\begin{Definition}
Let $(X, \mathcal{B}, \mu, T)$ be a measure preserving transformation. 
A decreasing sequence $X = S_0 \supset S_2 \supset \cdots$ of measurable sets with $\mu(S_n) \rightarrow 0$ is called a {\em separating sieve} if there exists a set $M$ of measure zero with the following property: if $x, y \in X \setminus M$ and if for each $N$ there exists $n$ such that $T^n x, T^n y \in S_N$, then $x=y$.
A measure preserving system $(X, \mathcal{B}, \mu, T)$ is called {\em measurably distal} if it has a separating sieve.
\end{Definition}

It is not hard to see that if $(X,T)$ is distal, then for any $T$-invariant Borel probability measure $\mu$ on $(X, \mathcal{B})$,  $(X, \mathcal{B}, \mu, T)$ is measurably distal.
In particular, any uniquely ergodic\footnote{A topological dynamical system $(X,T)$ is called {\em uniquely ergodic} if there is a unique $T$-invariant Borel probability measure $\mu$.} distal transformation such as, say, $S(x,y) = (x+ \alpha, y + a x + \beta)$, where $\alpha \notin \mathbb{Q}, a \in \mathbb{Z}\setminus\{0\}$ and $\beta \in \mathbb{R}$, is measurably distal. 
This observation will be utilized in Corollary \ref{ex:skewprod} below.

\begin{prop}
\label{prop:dis}
Let $ \mathsf{X}_i = (X_i, \mathcal{B}_i, \mu_i, T_i), 1 \leq i \leq k,$ be ergodic, measurably distal systems. If Kronecker factors\footnote{ {\em Kronecker factor} of an ergodic system is the largest factor of the system that is isomorphic to a rotation on a compact abelian group.} of $T_i$ are mutually disjoint, then $T_1, \dots, T_k$ are mutually disjoint.
\end{prop}
\begin{proof}
The proof goes by induction.

Berg (\cite{KBerg1}, \cite{KBerg2}) showed that if $\mathsf{X}$ is an ergodic system and $\mathsf{Y}$ is an ergodic and measurably distal system, then
$\mathsf{X}$ and $\mathsf{Y}$ are disjoint if and only if the Kronecker factors of $\mathsf{X}$ and $\mathsf{Y}$ are disjoint. (See also Theorem 1.1 in \cite{MoRiRo}.)

By Berg's result,  $\mathsf{X}_1$ and $\mathsf{X}_2$ are disjoint. Let $i \geq 2$. We will prove that if $\mathsf{X}_1, \dots, \mathsf{X}_i$ are mutually disjoint, then $\mathsf{X}_1, \dots, \mathsf{X}_{i+1}$ are mutually disjoint. 

Let $\nu$ be a $T_1 \times \cdots \times T_{i+1}$-invariant measure such that
for any $j = 1, 2, \dots, i+1$ the projection of $\nu$ onto $X_j$ equals $\mu_j$. Then the projection of $\nu$ on $X_1 \times \cdots \times X_i$ is $\mu_1 \times \cdots \times \mu_i$. Moreover, $\mathsf{X}_1, \dots, \mathsf{X}_i$ are mutually disjoint, so $\mathsf{X}_1 \times \cdots \times \mathsf{X}_i$ is ergodic by Lemma \ref{lem:dis:prod:erg}. 
Applying Berg's result to $\mathsf{X} = \mathsf{X}_1 \times \cdots \times \mathsf{X}_i$ and $\mathsf{Y} =  \mathsf{X}_{i+1}$ yields that $\nu = \mu_1 \times \cdots \times \mu_{i+1}$.
\end{proof}

The following immediate corollary of Proposition \ref{prop:dis}  will be used below in Proposition \ref{ex:dijoint:dim2}. 

\begin{Corollary}
\label{ex:skewprod}
Consider the following measure preserving transformations on $(\mathbb{T}^2, \mathcal{B}, \lambda)$, where $\mathbb{T}^2$ is the $2$-dimensional torus, $\mathcal{B}$ is the Borel $\sigma$-algebra  and $\lambda$ is the Lebesgue measure:
\begin{equation*}
    S_j(x,y) = (x+ \alpha_j, y + a_j x + \beta_j ),
\end{equation*}
where $\alpha_j \notin \mathbb{Q}$, $a_j \in \mathbb{Z} \setminus \{0\}$ and  $\beta_j \in \mathbb{R}$ for $i=1, 2, \dots, k$. 

If $1, \alpha_1, \alpha_2, \cdots, \alpha_k$ are rationally independent, then $S_1, \dots, S_k$ are mutually disjoint.
\end{Corollary}

\begin{proof}
In view of Proposition \ref{prop:dis}, the result follows from the fact that each $S_j$ is distal, ergodic and its Kronecker factor is given by the projection on the first coordinate (cf. \cite{Anz}, Corollary of Theorem 6.)
\end{proof}


The following theorem provides a convenient criterion for uniform joint ergodicity.  
\begin{Theorem}
\label{Thm:disjoint}
Let $(X, \mathcal{B}, \mu)$ be a probability space. Suppose that $\mathsf{X}_i^{(1)} = (X, \mathcal{B}, \mu_i, S_i)$ ($i= 1,  2, \dots, l$) and $\mathsf{X}_j^{(2)} = (X, \mathcal{B}, \nu_j, T_j)$ ($j=1, 2, \dots, m$) are measure preserving systems such that $\mu_i, \nu_j$ are equivalent to $\mu$ and 
\begin{enumerate}
\item $S_1, \dots, S_l$ are uniformly jointly ergodic with respect to $(\mu; \mu_1, \dots, \mu_l)$: 
 for any $f_1, \dots, f_l \in L^{\infty}$, 
$$\lim\limits_{N -M \rightarrow \infty} \frac{1}{N -M } \sum\limits_{n=M}^{N-1}  S_1^{n} f_1 \cdot S_2^n f_2 \cdots S_l^n f_l = \prod_{i=1}^l \int f_i \, d \mu_i \quad \text{ in } L^2(\mu).$$ 
\item $T_1, \dots, T_m$ are uniformly jointly ergodic with respect to $(\mu; \nu_1, \dots, \nu_m)$:
 for any $g_1, \dots, g_m \in L^{\infty}$, 
$$\lim\limits_{N -M \rightarrow \infty} \frac{1}{N -M } \sum\limits_{n=M}^{N-1}   T_1^{n} g_1 \cdot T_2^n g_2 \cdots T_m^n g_m = \prod_{j=1}^m \int g_j \, d \nu_j \quad \text{ in } L^2(\mu).$$ 
\end{enumerate}
If the product systems $\mathsf{X}^{(1)} = \mathsf{X}_1^{(1)} \times \cdots \times \mathsf{X}_l^{(1)}$ and $\mathsf{X}^{(2)} = \mathsf{X}_1^{(2)} \times \cdots \times \mathsf{X}_m^{(2)}$ are disjoint, then transformations $S_1, \dots, S_l, T_1, \dots, T_m$ are uniformly jointly ergodic with respect to $(\mu; \mu_1, \dots, \mu_l, \nu_1, \dots, \nu_m)$:
 for any $f_1, \dots, f_l, g_1, \dots, g_m \in L^{\infty}$, 
$$\lim\limits_{N -M \rightarrow \infty} \frac{1}{N -M } \sum\limits_{n=M}^{N-1}   S_1^{n} f_1 \cdots S_l^n f_l \cdot T_1^n g_1 \cdots T_m^n g_m = \prod_{i=1}^l \int f_i \, d \mu_i  \cdot \prod_{j=1}^m \int g_j \, d \nu_j \quad \text{ in } L^2(\mu).$$
\end{Theorem}

\begin{proof}
 Our assumptions imply that $\mathsf{X}^{(1)}$ and $\mathsf{X}^{(2)}$ are ergodic and so by Lemma \ref{lem:dis:prod:erg}, $\mathsf{X}^{(1)} \times \mathsf{X}^{(2)}$ is an ergodic system.

By Theorem \ref{criterion:JE}, it is sufficient to show that for any $A_1, \dots, A_l, B_1, \dots, B_m \in \mathcal{B}$
\begin{equation}
\label{eqn:6.1}
\begin{split}
    \lim_{N-M \rightarrow \infty} \frac{1}{N - M} \sum_{n=M}^{N-1} &\mu \left( S_1^{-n} A_1 \cap \cdots \cap S_l^{-n} A_l \cap T_1^{-n} B_1 \cap \cdots \cap T_m^{-n} B_m \right)  \\
    &= \prod_{i=1}^l \mu_i(A_i) \cdot \prod_{j=1}^m \nu_j(B_j).
\end{split}
\end{equation}

Consider the family of measures $\rho_{M,N}$, $M, N \in \mathbb{N} \cup \{0\}$, $0 \leq M < N$,  defined by \begin{equation*}
    \begin{split}
         \rho_{M, N} &(A_1 \times \cdots \times A_l \times B_1 \times \cdots \times B_m) \\
&:= \frac{1}{N - M} \sum_{n=M}^{N-1} \mu \left( S_1^{-n} A_1 \cap \cdots \cap S_l^{-n} A_l \cap T_1^{-n} B_1 \cap \cdots \cap T_m^{-n} B_m \right),
    \end{split}
\end{equation*}
for any $A_1, \dots, A_l, B_1, \dots, B_m \in \mathcal{B}$.

We claim that any limit point of $\rho_{M,N}$ is $\mu_1 \times \cdots \times \mu_l \times \nu_1 \times \cdots \times \nu_m$ (which clearly will imply \eqref{eqn:6.1}).
Note that since $S_1, \dots, S_l$ are uniformly jointly ergodic and $T_1, \dots, T_m$ are uniformly jointly ergodic, we have that projections of any limit point of $\rho_{M, N}$ to $\mathsf{X}^{(1)}$ and $\mathsf{X}^{(2)}$ are $\mu_1 \times \cdots \times \mu_l$ and $\nu_1 \times \cdots \times \nu_m$. 
 Our claim now follows from the fact that $\mathsf{X}^{(1)}$ and $\mathsf{X}^{(2)}$ are disjoint. 
\end{proof}


Note that in the special case when $l=m=1$ and $\mu_1 = \mu_2 = \mu$, Theorem \ref{Thm:disjoint} implies that if $T$ and $S$ are ergodic and disjoint transformations on $(X, \mathcal{B}, \mu)$, then $T$ and $S$ are uniformly jointly ergodic.
A more general result, proved in \cite{HKS}, states that, for any $k \geq 2$, if ergodic measure preserving transformations $T_1, \dots, T_k$ on $(X, \mathcal{B}, \mu)$ are disjoint, then they are uniformly jointly ergodic. The following corollary of Theorem \ref{Thm:disjoint} provides a slight generalization of this result.

\begin{Corollary}
\label{Cor:HKS}
Let $(X, \mathcal{B}, \mu)$ be a probability space. 
Let $k \in \mathbb{N}$.
If ergodic measure preserving systems $(X, \mathcal{B}, \mu_i, T_i), 1 \leq i \leq k,$ are {\em mutually disjoint} and each $\mu_i$ is equivalent to $\mu$, then $T_1, \dots, T_k$ are uniformly jointly ergodic with respect to $(\mu; \mu_1, \cdots, \mu_k)$. 
\end{Corollary}

\begin{proof}
The result is trivial for $k=1$. Assume that the result in question holds for $T_1, \dots, T_{k-1}$. By assumption, $T_1 \times \cdots \times T_{k-1}$ and $T_{k}$ are disjoint. Thus, Theorem \ref{Thm:disjoint} implies that 
$T_1, \dots, T_{k}$ are uniformly jointly ergodic with respect to $(\mu; \mu_1, \cdots, \mu_k)$.
\end{proof}



The following corollary of Theorem \ref{Thm:disjoint} deals with joint ergodicity of zero entropy systems and exact systems. 
\begin{Corollary}
\label{Cor:dis}
Let $(X, \mathcal{B}, \mu)$ be a probability space.
Suppose that $S_1, \dots, S_l, T_1, \dots, T_m: X \rightarrow X$ are measurable transformations such that 
\begin{itemize}
    \item each $S_i$, $1 \leq i \leq l$, has an invariant measure $\mu_i$, which is equivalent to $\mu$, and has zero entropy,
    \item each $T_i$, $1 \leq i \leq m$, has an invariant measure $\nu_i$,  which is equivalent to $\mu$, and is exact.
\end{itemize}
If $ S_1, \dots, S_l$ are uniformly jointly ergodic with respect to $(\mu; \mu_1, \dots, \mu_l)$ and $T_1, \dots, T_m$ are uniformly jointly ergodic with respect to $(\mu; \nu_1, \dots, \nu_m)$, then $S_1, \dots, S_l, T_1, \dots, T_m$ are uniformly jointly ergodic with respect to $(\mu, \mu_1, \dots, \mu_l, \nu_1, \dots, \nu_m)$.
\end{Corollary}


\begin{proof}
Since any product of zero entropy systems has zero entropy, $S_1 \times \cdots \times S_l$ has zero entropy. 

Note that natural extensions of exact transformations are K-automorphisms (\cite{Rok}) and   products of K-automorphism are K-automorphisms. Thus, $T_1 \times \cdots \times T_m$ is a factor of a K-automorphism.  

It is well known that K-automorphisms are disjoint from zero entropy systems (see, for instance, Theorem 2 in \cite{delaRue}). Also, if $\mathsf{X}$ and $\mathsf{Y}$ are disjoint and $\mathsf{Z}$ is a factor of $\mathsf{X}$, then $\mathsf{Z}$ and $\mathsf{Y}$ are disjoint (Proposition I.1 in \cite{Fur}). Thus, $S_1 \times \cdots \times S_l$ and $T_1 \times \cdots \times T_m$ are disjoint. Corollary \ref{Cor:dis} follows now from Theorem \ref{Thm:disjoint}. 
\end{proof}

The following example is an immediate consequence of Corollary \ref{Cor:dis}, since $T_G$ and $T_{\beta}$ are exact.
\begin{ex}
\label{ex:notThm4.1}
Let $\mu_0$ be a probability measure on $[0,1]$, which is equivalent to Lebesgue measure $\lambda$. 
If a $\mu_0$-preserving transformation $S: [0,1] \rightarrow [0,1]$ is ergodic and has zero entropy, 
then for any bounded measurable functions $f_0, f_1, f_{2,1}, \dots, f_{2,s}$, 
\begin{equation}
\label{eq:ex5.5ext}
\begin{split} 
\begin{split} 
\lim\limits_{N -M \rightarrow \infty} \frac{1}{N -M } \sum\limits_{n=M}^{N-1} 
&  f_{0} (S^n x) \cdot f_1 (T_G^n x) \cdot \prod_{i=1}^s f_{2,i}(T_{\beta_i}^n x) \\
 &=  \int f_{0}  \, d \mu_0 \cdot  \int f_1 \, d \mu_G \cdot \prod_{i=1}^s \int f_{2,i} \,  d \mu_{\beta_i} \,\, \text{in } L^2(\lambda),
\end{split}
\end{split}
\end{equation}
where $\beta_1 , \cdots , \beta_t$ are distinct real numbers with $\beta_i >1$ and $\log \beta_i \ne \frac{\pi^2}{6 \log 2}$. 
\end{ex}

\begin{Remark}
\label{Rem:disjoint:ex}
 Note that the transformation $S$, which appears in Example \ref{ex:notThm4.1}, may not belong to $\mathcal{T}$. Hence Example \ref{ex:notThm4.1} is not covered by Theorem \ref{thm:L^2_joint}.
\end{Remark}

It is easy to see that the term $f_0(S^n x)$ appearing in the right hand side of \eqref{eq:ex5.5ext} can be replaced by more general expression involving any finite family of ergodic mutually disjoint transformations having zero entropy. For example, 
 let $S_0: [0,1] \rightarrow [0,1]$ be von Neumann - Kakutani transformation (see \eqref{form:VK}), let $S_i x = (x + \alpha_i) \, \bmod 1, (1 \leq i \leq l)$, where $1, \alpha_1, \dots, \alpha_l$ are rationally independent, and let $S_{l+1}:[0,1] \rightarrow [0,1]$ be Chacon's transformation (see the Appendix).
Then, Corollary \ref{Cor:dis} implies that for any bounded measurable functions $f_{0}, \dots, f_{l+1}, g, h_1, \dots h_t$, 
\begin{equation}
\label{eq8.4}
\begin{split} 
\lim\limits_{N -M \rightarrow \infty} \frac{1}{N -M } \sum\limits_{n=M}^{N-1} 
&  \prod_{i=0}^{l+1} f_{i} (S_i^n x) \cdot g(T_G^n x) \cdot \prod_{i=1}^t h_{i}(T_{\beta_i}^n x) \\
 &= \prod_{i=0}^{l+1} \int f_{i} \, d \lambda \cdot  \int g \, d \mu_G \cdot \prod_{i=1}^t \int h_{i} \, d \mu_{\beta_i} \,\, \text{in } L^2(\lambda).
\end{split} 
\end{equation}

\begin{Remark}
\label{Rem:ex6.2}
Berend  proved in \cite{Be} that if commuting, invertible measure preserving transformations $T_1, \dots, T_k$ on a probability space $(X, \mathcal{B}, \mu)$ are jointly ergodic, then each transformation $T_i$ is totally ergodic for $i =1, 2, \dots, k$ and, moreover, $T_1, \dots, T_k$ are totally jointly ergodic, that is, $T_1^m, \dots T_k^m$ is jointly ergodic for any $m \in \mathbb{N}$. If we drop the condition of commutativity, then total joint ergodicity may not hold.
For instance, if $T_1$ is von Neumann - Kakutani transformation and $T_2$ is an irrational rotation (note that $T_1$ and $T_2$ do not commute), then by \eqref{eq8.4} $T_1$ and $T_2$ are jointly ergodic, yet $T_1$ is not totally ergodic. 
\end{Remark} 

In this paper we mainly deal with transformations on the unit interval. However, Theorem \ref{Thm:disjoint} has a wider range of applications. We illustrate this by the following proposition.
\begin{prop}
\label{ex:dijoint:dim2}
Consider the mutually disjoint skew-product transformations on $\mathbb{T}^2$ introduced in Example \ref{ex:skewprod}
$$S_i(x,y) = (x+ \alpha_i, y + a_i x + \beta_i ), 1 \leq i \leq l,$$
where $1, \alpha_1, \dots, \alpha_l$ are rationally independent. 
Let $T_1, \dots, T_m \in SL(2, \mathbb{Z})$ be hyperbolic automorphisms $\mathbb{T}^2$ with $h(T_1) < h(T_2) < \cdots < h(T_m)$.  
Then $S_1, \dots, S_l, T_1, \dots, T_m$ are uniformly jointly ergodic on $(\mathbb{T}^2, \mathcal{B}, \lambda)$.  
\end{prop}

\begin{proof}
Note first that by Corollary \ref{Cor:HKS}, $S_1, \dots, S_l$ uniformly jointly ergodic. Moreover, $h(S_i) = 0$ for all $i = 1, 2, \dots, l$ (This follows, for example, from the fact that the transformations $S_i$ are distal. See \cite{Pa3}). 

Now, it follows from \cite{BergGor}, Proposition 2.10, that if hyperbolic $T_1, \dots, T_m \in SL(2, \mathbb{Z})$ have distinct entropy, then they are jointly mixing, hence uniformly jointly ergodic on $(\mathbb{T}^2, \mathcal{B}, \lambda)$. (See Remark \ref{rem:mixje}.) 

Note also that $T_i$ are Bernoulli (see \cite{Katz}), so $T_1 \times \cdots \times T_m$ is a Bernoulli transformation, which is disjoint from the zero entropy transformation $S_1 \times \cdots \times S_l$. Now the result in question follows from Theorem \ref{Thm:disjoint}.
\end{proof}

\section{Miscellany}
\label{sec:final remark}

The goal of this section is to provide some results and formulate open questions which are related to possible extensions of results presented in this paper. 

\subsection{Piecewise monotone maps not in $\mathcal{T}$.}
\mbox{}

Most of theorems in this paper involve (and apply to) transformations belonging to the class $\mathcal{T}$.  
By the definition of the class $\mathcal{T}$, if $T$ belongs to $\mathcal{T}$, then there exist a $T$-invariant measure $\mu$ which is ``strongly equivalent" to Lebesgue measure in the following sense: 
\begin{equation}
\label{strongequiv}
\frac{1}{C} \leq \frac{d \mu}{ d \lambda} \leq C \quad \text{for some } C \geq 1.
\end{equation}

However, there are some interesting piecewise monotone maps which have an invariant measure that is equivalent to Lebesgue measure, but does not satisfy \eqref{strongequiv}. Consider, for example,  Ulam - von Neumann map $T: [0,1] \rightarrow [0,1]$, which is defined by
\[ Tx = 4x (1-x).\]
Note that $T$ is $\mu$-preserving, where $d \mu = \frac{1}{\pi \sqrt{x(1-x)}} dx$.
Define the transformation $S: [0,1] \rightarrow [0,1]$ by
\[S x =     9x - 24 x^2 + 16x^3. \]
It is not hard to check that $S$ is $\mu$-preserving and, moreover, $T$ and $S$ are isomorphic to $\tilde{T} x =2x \, \bmod 1$ and $\tilde{S} x= 3x \, \bmod 1$ respectively via the isomorphism $\psi(x) = \sin^2 \frac{\pi x}{2}$. So we have the following result. 
\begin{Theorem}
$T$ and $S$ are uniformly jointly ergodic: for any $f_1, f_2 \in L^{\infty}$, 
\begin{equation}
\frac{1}{N} \sum_{n=0}^{N-1} f_1 (T^n x) f_2(S^n x) \rightarrow \int f_1 d \mu \int f_2 d \mu \quad \text{in } L^2.
\end{equation}
\end{Theorem}

The above discussion motivates the following question.
\begin{Question}
Can Theorem \ref{thm:L^2_joint} be extended to the class $\tilde{\mathcal{T}}$, which is introduced in the following definition?
\end{Question} 
\begin{Definition}
Let $(X, \mathcal{B}, \mu, T)$ be an ergodic measure preserving system such that $\mu$ is equivalent to Lebesgue measure $\lambda$ and let $\mathcal{A}$ be a partition of $[0,1]$. We will say that $(T, \mu, \mathcal{A}) \in \tilde{\mathcal{T}}$ if 
\begin{enumerate}
\item 
for any interval $I \in \mathcal{A}$, $T|_{\text{int}(I)}$ is continuous and strictly monotone,
\item  partition $\mathcal{A}$ generates the $\sigma$-algebra $\mathcal{B}$, meaning that $\mathcal{B} = \bigvee_{j=0}^{\infty} T^{-j} \sigma( \mathcal{A}) \mod \lambda$, where $\sigma (\mathcal{A})$ denotes the sub-$\sigma$-algebra generated by $\mathcal{A}$,
\item entropy of the partition $\mathcal{A}$ is finite: $H(\mathcal{A}) = - \sum \mu(A_i) \log \mu (A_i) < \infty$. 
\end{enumerate}
\end{Definition}

\subsection{Joint ergodicity of transformations on $\mathbb{T}^d \, (d \geq 2)$}
\mbox{}

In this article, we mainly discussed transformations on the unit interval. It is certainly of interest to obtain more general versions of theorems proved in Sections \ref{sec:jointergodicityPM} and \ref{sec:jointmixingPM}.
The following examples demonstrate some new phenomena which one meets when dealing with joint ergodicity on tori of higher dimension. 

\begin{ex}
Consider the following endomorphisms of $\mathbb{T}^2$: 
\[
S_1 = 
\begin{bmatrix}
2 & 0 \\
0 & 3 
\end{bmatrix}
\quad 
S_2 = 
\begin{bmatrix}
3 & 0 \\
0 & 2 
\end{bmatrix}
\quad 
S_3 = 
\begin{bmatrix}
2 & 0 \\
0 & 4 
\end{bmatrix}
\]

Clearly  $h(S_1) = h(S_2)$. Notice that $S_1 = T_2 \times T_3$ and $S_2 = T_3 \times T_2$, where $T_2$ is times $2$ map and $T_3$ is times $3$ map. It is not hard to verify $S_1$ and $S_2$ are uniformly jointly ergodic (indeed jointly mixing).
On the other hand, $S_1$ and $S_3$ have different entropies but they are not jointly ergodic since the projections of $S_1$ and $S_3$ to the first coordinate coincide. 
\end{ex}

\begin{ex}
Consider now the following automorphisms of $\mathbb{T}^2$:
\[
S_1 = 
\begin{bmatrix}
2 & 1 \\
1 & 1 
\end{bmatrix}
\quad 
S_2 = 
\begin{bmatrix}
1 & 1 \\
1 & 2 
\end{bmatrix}
\quad 
S_3 = 
\begin{bmatrix}
3 & 1 \\
-1 & 0 
\end{bmatrix}
\]
Note that $h(S_1) = h(S_2) = h(S_3)$. It follows from \cite[Proposition 2.10]{BergGor},  that $S_1, S_2, S_3$ are not jointly ergodic, while any two $S_i, S_j, i \ne j,$ are jointly mixing.
\end{ex}

\begin{ex}
Recall that the Baker map and its inverse are defined by the formulas
$$T_B(x,y) = 
\begin{cases} 
      (2x, y/2), & 0 \leq x < 1/2,  \\
      (2x-1, y/2 + 1/2), &  1/2 \leq x <1,
   \end{cases}
$$  
and
$$T_{B}^{-1}(x,y) = 
\begin{cases} 
      (x/2, 2y), & 0 \leq y < 1/2,  \\
      (1/2+ x/2, 2y-1), &  1/2 \leq y <1.
   \end{cases}
$$
One can check that 
\begin{enumerate}
\item $T_B$ and $T_B^{-1}$ are uniformly jointly ergodic (indeed, jointly mixing) since $T_B$ is mixing of all orders. However, $h(T_B) = h(T_B^{-1}) = \log 2$.

\item $T_B$ and $T_2 \times T_2$ have the same actions on the functions of the form $f(x,y) = g(x)$, and so they are NOT jointly ergodic, although they have different entropies.
\end{enumerate}
\end{ex}

\subsection{Sequential version of joint mixing}
\mbox{}

In this subsection we formulate a natural sequential version of joint mixing (Theorem \ref{Thm:JM2:new}) which leads to an interesting application and to a natural question. 
The proof is analogous to that of Theorem \ref{Thm:JM2:new} and is omitted. 

\begin{Theorem}
\label{JM:MainThm}
For $i = 0, 1, 2, \dots, k$, let $\mu_i$ be a probability measure on the measurable space $([0,1], \mathcal{B})$ and let  $T_i:[0,1] \rightarrow [0,1]$ be $\mu_i$-preserving ergodic transformations. 
Suppose that
\begin{enumerate}
\item for $i =0, 1, 2, \dots, k$, there is a partition $\mathcal{A}_i$ with  $(T_i, \mu_i, \mathcal{A}_i) \in \mathcal{T}$,
\item for $i = 1, 2, \dots, k$, $(T_i, \mu_i, \mathcal{A}_i)$ is $\alpha$-mixing.
\end{enumerate} 
Let $h(T_i)$ be entropy of $T_i$ and let $(a_n^{(i)})$ be an increasing sequence of natural numbers. Suppose that $h(T_i) < \infty$ for all $i$ and assume that there is $\epsilon_0 >0$ such that for $i = 0, 1, 2, \dots, k-1$,
\begin{equation}
\label{jmcond}
\lim_{n \rightarrow \infty} (a_n^{(i+1)} (h(T_{i+1}) - \epsilon_0) - a_n^{(i)} ( h(T_i) + \epsilon_0)) = \infty.
\end{equation} 
If $T_0$ is mixing, then for any $B, A_0, A_1, \dots, A_k \in \mathcal{B}$,
\begin{equation}
\label{proof:goal:JM2}
 \lim_{n \rightarrow \infty} \nu (B \cap T_0^{-a_n^{(0)}} A_0 \cap T_1^{-a_n^{(1)}}A_1 \cap \cdots \cap T_k^{-a_n^{(k)}}A_k) =  \nu (B) \prod_{i=0}^k \mu_i(A_i),
\end{equation}
where $\nu$ is any probability measure equivalent to the Lebesgue measure $\lambda$.
\end{Theorem}

 Note that in the following Corollary, we do not require that $\beta_1, \dots, \beta_k$ are distinct and $\log \beta_i \ne \frac{\pi^2}{6 \log 2}$. 
\begin{Corollary}
Let $(a_n^{(i)})$ be an increasing sequence of natural numbers for $i = 0, 1, \dots, k$, such that $\lim\limits_{n \rightarrow \infty} \frac{a_n^{(i+1)}}{a_n^{(i)}} = \infty$ for $i = 0, 1, 2, \dots, k-1$. 
Let $\beta_1 , \cdots , \beta_k$ be real numbers with $\beta_i >1$. 
If  $\nu$ is any probability measure equivalent to the Lebesgue measure $\lambda$, then for any measurable sets $ B, A_0, A_1 \dots, A_k$, 
\begin{equation*}
\lim\limits_{n \rightarrow \infty} \nu (B \cap T_G^{-a_n^{(0)}} A_0 \cap T_{\beta_1}^{-a_n^{(1)}} A_1 \cap \cdots \cap T_{\beta_k}^{-a_n^{(k)}} A_k) = \nu (B) \, \mu_G (A_0) \, \prod_{i=1}^k \mu_{\beta_i} (A_i), 
\end{equation*}
where $T_G (x) = \frac{1}{x} \, (\bmod \, 1) $ is Gauss map and $T_{\beta_i} (x) = \beta_i x \, (\bmod \, 1)$ is $\beta$-transformation for $1 \leq i \leq k$.
\end{Corollary} 

We suspect that $\epsilon_0$ in the formulation \eqref{jmcond} can be removed.
\begin{Question}
Does Theorem \ref{JM:MainThm} hold if the condition \eqref{jmcond} is replaced with the following one:
\begin{equation*}
\lim_{n \rightarrow \infty} ( a_n^{(i+1)} h(T_{i+1}) - a_n^{(i)} h(T_i) ) = \infty \,\, ?
\end{equation*} 
\end{Question}

\section{Appendix: Rank one transformations on the interval $[0,1]$}
\label{Sec:App}
In this Appendix, we present a constructive geometric definition of rank one transformations (see \cite{Fer}, Definition 4).
This will allow us to explain why von Neumann-Kakutani transformation, Chacon transformation and Smorodinsky-Adams transformation mentioned in Subsection \ref{subsec:rank1} are rank one and belong to the class $\mathcal{T}$.

\begin{Definition}
A measure preserving system $(X, \mathcal{B}, \mu, T)$ is called {\em rank one} if there exist a sequences of positive integers $q_n$, $n \in \mathbb{N} \cup \{0\}$, with $q_n >1$ for infinitely many $n$, and nonnegative integers $s_{n, i}, n \in \mathbb{N} \cup \{0\}, 0 \leq i \leq q_{n-1}$, such that, if $h_n$ are defined by 
\begin{equation}
h_0 = 1, \,\, h_{n+1} = q_n h_n + \sum_{i=0}^{q_n-1} s_{n,i}
\end{equation}
then 
\begin{equation}
\label{finite:rank}
 \sum_{n=0}^{\infty} \frac{h_{n+1} - q_n h_n}{h_{n+1}} < \infty,
\end{equation}
and subsets of $X$, denoted by $F_n, n \in \mathbb{N} \cup \{0\}$, $F_{n, i}, n \in \mathbb{N} \cup \{0\}, 0 \leq i \leq q_n -1$, and $S_{n, i, j}, n \in \mathbb{N} \cup \{0\}, 0 \leq i \leq q_n -1, 0 \leq j \leq s_{n,i} -1$, (if $s_{n, i} = 0$ there are no $S_{n, i,j}$), such that for all $n$:
\begin{itemize}
\item $(F_{n,i}, 0 \leq i \leq q_n -1)$ is a partition of $F_n$,
\item $T^k F_n, 0 \leq k \leq h_n-1,$ are disjoint,
\item $T^{h_n} F_{n,i} = S_{n,i,1}$ if $s_{n,i} \ne 0$,
\item $T^{h_n} F_{n,i} = F_{n, i+1}$ if $s_{n,i} = 0$ and $0 < i < q_n-1$,
\item $TS_{n,i,j} = S_{n,i,j+1}$ if $j < s_{n,i} -1$,
\item $TS_{n, i, s_{n, i-1}} = F_{n, i+1}$ if $i < q_n-1$,
\item $F_{n+1} = F_{n,0}$,
\end{itemize}
and the partitions $P_n = \{ F_n, T F_n, \dots, T^{h_n -1} F_n, X \setminus \bigcup\limits_{k=0}^{h_n -1} T^k F_n \}$ are increasing to $\mathcal{B}$.
\end{Definition}

Note the above condition \eqref{finite:rank} guarantees that $\mu(X) < \infty$.
By choosing $r = \mu(F_0)$ appropriately, we can guarantee $\mu(X) = 1$.
Since the main focus in this article is interval maps, it is convenient to assume that $X = [0,1]$ and $\mu = \lambda$. Moreover, we may and will assume that the sets $F_{n,i}$ and $S_{n,i,j}$ are intervals and that transformation $T$ on $X = [0,1]$ is piecewise linear.

At stage $n$, we have a Rokhlin tower $\tau_n = \{F_n, T F_n, \dots, T^{h_n -1} F_n\}$ with base $F_n$ and height $h_n$. We may assume that $\tau_n$ consists of intervals having equal length and $T$ is a piecewise linear map sending $T^i F_n$ to $T^{i+1} F_{n}$.  At stage $n+1$, Rokhlin tower $\tau_{n+1}$  is constructed from $\tau_n$ by cutting into $q_n$ columns of equal width and stacking them above each other with the possible addition of intervals $S_{n,i,j}$ between the columns.
So, $\tau_{n+1}$ consists of $h_{n+1} (= q_n h_n + \sum_{i=0}^{q_n -1} s_{n,i})$ intervals, and, at stage $n+1$, $T$ is defined on all the intervals in $\tau_{n+1}$ except the last interval $T^{h_{n+1} -1} F_{n+1}$.
In this recursive process, we have that if $r = |F_0|$ is the length of the interval $F_0$, then $|F_{n}| = \frac{r}{q_0 \cdots q_{n-1}}$ for $n \geq 1$.

If for some $0 \leq i \leq h_{n+1}-2$, the interval $T^i F_{n+1}$ is a subset of $T^j F_n$ for some $0 \leq j \leq h_n -2$, then $T$ was already defined on $T^i F_{n+1}$ before stage $n+1$ and the number of such intervals is $q_n(h_n -1)$.
Therefore, at stage $n+1$, there are $h_{n+1} -1 - q_n(h_n -1)$ intervals $I_{n,j}, 1 \leq j \leq  h_{n+1} -1 -  q_n (h_n -1),$ in Rokhlin tower $\tau_{n+1}$, with the property that $T$ is defined on these intervals at stage $n+1$, but was not defined before stage $n+1$. 
Note that $|I_{n,j}| = |F_{n+1}| = \frac{r}{q_0 q_1 \cdots q_n}$ for any $j = 1, 2, \dots, h_{n+1} -1 -  q_n (h_n -1)$. 
Consider now the family $\mathcal{A} = \{ I_{n,j}: n = 1, 2, \cdots, 1 \leq j \leq  h_{n+1} -1 -  q_n (h_n -1) \}$, which forms a countable partition of $[0,1]$.
We have
\begin{align}
\label{rank1:FiniteEntropy}
   H (\mathcal{A}) &= - \sum_{n=0}^{\infty} (h_{n+1} -1 - q_n (h_n -1)) \frac{r}{q_0 \cdots q_n} \log \frac{r}{q_0 \cdots q_n} \notag \\
 &= - \sum_{n=0}^{\infty} \left(\sum_{i=0}^{q_n -1} s_{n,i} + q_n -1 \right) \frac{r}{q_0 \cdots q_n} \log \frac{r}{q_0 \cdots q_n}.
\end{align}
Note that $T$ is a piecewise linear map and so if the parameters $q_n$ and $s_{n,i}$ are such that the quantity in \eqref{rank1:FiniteEntropy} is finite, then $(T, \lambda, \mathcal{A}) \in \mathcal{T}$.

The following list specifies the parameters $q_n$ and $s_{n,i}$ for the examples presented in Subsection \ref{subsec:rank1}. 
\begin{enumerate}
\item Von Neumann-Kakutani transformation: $q_n = 2$ and $s_{n,i} = 0$.  
\item Chacon transformation: $q_n =3$ and $s_{n,i} = 1$ if $i=1$ and $0$ otherwise.  
\item Smorodinsky-Adams transformation: $q_n = n$ and $s_{n,0} = 0 , s_{n,i} = i-1$.  
\end{enumerate}
It is not hard to check that the choice of parameters $q_n$ and $s_{n,i}$ in each of the above examples guarantee that the quantity in \eqref{rank1:FiniteEntropy} is finite.

\end{document}